\documentclass{amsart}
\usepackage{amssymb,amsthm}
\usepackage{hyperref}
\usepackage{thmtools}
\usepackage{thm-restate}
\usepackage{todonotes}
\usepackage[T1]{fontenc}
\usepackage[utf8]{inputenc}
\usepackage{lmodern}
\usepackage{thmtools}
\usepackage[capitalize]{cleveref}
\usepackage{xspace}

\theoremstyle{plain}
\newtheorem{thm}{Theorem}[section]

\newtheorem{prop}[thm]{Proposition}
\newtheorem{lem}[thm]{Lemma}
\newtheorem{cor}[thm]{Corollary}

\theoremstyle{definition}
\newtheorem{defn}[thm]{Definition}

\theoremstyle{remark}

\numberwithin{equation}{section}

\renewcommand\epsilon{\varepsilon}

\DeclareMathOperator\Mod{{\mathrm Mod}}
\DeclareMathOperator\Th{{\mathrm Th}}

\crefname{thm}{Theorem}{Theorems}
\def\M{{\mathcal M}}
\newcommand{\mc}[1]{\mathcal{#1}}
\newcommand{\Inv}[2]{{#2}^{(-#1)}}

\renewcommand{\phi}{\varphi}
\newcommand{\Pinf}[1]{\Pi^{\mathrm{in}}_{#1}}
\newcommand{\Pinc}[1]{\Pi^{\mathrm{c}}_{#1}}
\newcommand{\Sinf}[1]{\Sigma^{\mathrm{in}}_{#1}}
\newcommand{\Sinc}[1]{\Sigma^{\mathrm{c}}_{#1}}
\mathchardef\mhyph="2D
\providecommand{\dSinf}[1]{d\mhyph\Sigma^{\mathrm{in}}_{#1}}
\newcommand{\dSinc}[1]{d\mhyph\Sigma^{\mathrm{c}}_{#1}}
\newcommand{\Wwedge}{%
  \mathop{
    \mathchoice{\bigwedge\mkern-15mu\bigwedge}
               {\bigwedge\mkern-12.5mu\bigwedge}
               {\bigwedge\mkern-12.5mu\bigwedge}
               {\bigwedge\mkern-11mu\bigwedge}
    }
}

\newcommand{\Vvee}{%
  \mathop{
    \mathchoice{\bigvee\mkern-15mu\bigvee}
               {\bigvee\mkern-12.5mu\bigvee}
               {\bigvee\mkern-12.5mu\bigvee}
               {\bigvee\mkern-11mu\bigvee}
    }
}

\newcommand{\TBPT}{{\hyperref[thm:Inf4.3]{The Belligerent Pairs Theorem}}\xspace}

\begin{document}

\title{Infinite Belligerent Jump Inversion and Computable Scott Analysis}
\author{Uri Andrews}
\author{David Gonzalez} 
\author{Hongyu Zhu}

\address[Gonzalez]{University of Notre Dame,
Department of Mathematics,
Hurley Hall, 255 Hurley, Notre Dame, IN 46556,
  USA}
\address[Andrews, Zhu]{University of Wisconsin-Madison,
Department of Mathematics,
480 Lincoln Drive, Madison, WI, 53706, USA} 

\email[Andrews]{andrews@math.wisc.edu}
\email[Gonzalez]{dgonza42@nd.edu}
\email[Zhu]{hongyu@math.wisc.edu}
	\maketitle

\begin{abstract}

Scott analysis provides two fundamental tools for studying countable structures: Scott sentences, which characterize structures up to isomorphism, and back-and-forth relations, which measure structural similarity. A recurring phenomenon in computable structure theory is that many notions naturally associated with level $\alpha$ of Scott analysis have effective complexity at approximately $2\alpha$ jumps. This discrepancy appears both in the complexity of the back-and-forth relations and in the passage from arbitrary infinitary formulas to computable infinitary formulas.

We develop two new coding tools, the Belligerent Pairs Theorem and Belligerent Jump Inversion Theorem, which allow information at complexity level $2\alpha$ to be reflected in computable structures whose distinguishing features already appear at level $\alpha$. These results extend Harrison-Trainor's finite unfriendly jump inversion uniformly throughout the computable ordinals.

As applications, we determine the optimal interaction between syntactic complexity and oracle complexity for computable Scott sentences and for formulas distinguishing computable structures. For every computable infinite ordinal $\alpha$, we determine the oracle needed to compute a $\Pinf\alpha$ Scott sentence for a computable structure which has a $\Pinf\alpha$ Scott sentence. Any computable structure with a $\Pinf\alpha$ Scott sentence has a computable $\Pinf{2\alpha}$ Scott sentence. We show that both of these bounds are sharp. We prove analogous optimal results for formulas witnessing failure of the $\alpha$-back-and-forth relation. We also obtain further applications, including a resolution of a question of Chen, Gonzalez, and Harrison-Trainor concerning the complexity of back-and-forth classes.
\end{abstract}

{\small\noindent\textbf{\textit{Keywords-}} Scott rank, Back-and-forth relations, Scott complexity, Jump inversion}

\section{Introduction}

A fundamental aspect of mathematical practice is the ability to communicate and relate abstract concepts precisely.
To collaborate effectively, mathematicians must be able to unambiguously describe complex mathematical structures to one another, ensuring they are studying the same object.
Equally important is the ability to compare different structures: identifying underlying differences or determining when two seemingly distinct objects are essentially the same for a given application.

In infinitary logic (a type of logic allowing infinite conjunctions and disjunctions), these intuitive practices are rigorously formalized through Scott analysis.
A \textit{Scott sentence} acts as an absolute structural description, capturing a countable model so precisely that it uniquely identifies the structure up to isomorphism.
Meanwhile, \textit{back-and-forth relations} formalize the process of comparison.
By engaging in a stepwise matching game between the elements of two models, these relations provide a precise mathematical measure of their structural similarity, isolating their overlaps and differences step-by-step.

For example, the first mathematical structure most children are introduced to, the natural numbers, is described to us by a sentence of infinitary logic. In addition to defining the operations of addition and multiplication, to describe $\mathbb{N}$, we must say that all natural numbers are reached by counting or, written in infinitary logic,
\[\forall x(x=0\lor x=1\lor x=1+1\lor\cdots).\]
Later in life, we may learn about number fields, and observe that every root of a $\mathbb{Q}$-polynomial in $\mathbb{Q}[\sqrt2]$ is also in $\mathbb{Q}[\zeta_8]$ (but not vice-versa).
Logicians express this fact in terms of the back-and-forth relations $\mathbb{Q}[\sqrt2]\geq_1\mathbb{Q}[\zeta_8]$ and $\mathbb{Q}[\sqrt2]\not\leq_1\mathbb{Q}[\zeta_8]$.

Communication necessarily happens in the realm of the finite, but our descriptions allow for infinite formulas.
In the above examples, this is not an issue.
The ellipses in the formula describing the natural numbers have an understood meaning that allows for a finite description of the infinite object.
The difference between $\mathbb{Q}[\sqrt2]$ and $\mathbb{Q}[\zeta_8]$ is even simpler to communicate; it is the existence of a root for $x^8-1$.
A robust way to capture the finiteness in communication is to insist that the description of a structure or the difference between two structure is \textit{computable}.

Some of our main results concern the computability of both Scott sentences and the formulas witnessing differences between structures satisfying particular back-and-forth relations (\autoref{FullScottSent} and \autoref{FullBNF}).
We measure these computability-theoretic concerns in two important ways that interact with the syntactic complexity of the formulas.
First, we resolve the needed syntactic complexity of a Scott sentence or difference-witnessing formulas in the case where we insist that the formulas are computable.
Second, we determine the least non-computable oracle needed to describe a Scott sentence or difference-witnessing formula of optimal syntactic complexity. These results are summarized in the following two theorems, the first describing the complexity of describing a single structure, and the second describing the problem of distinguishing two structures.

\begin{defn}
	For any ordinal $\alpha$, we define $\gamma(\alpha)$ as the least ordinal such that $\gamma(\alpha)+\alpha\ge2\alpha$.
	Specifically, if $\alpha$ is finite, then $\gamma(\alpha)=\alpha$. If $\alpha$ is a limit ordinal, then $\gamma(\alpha)=0$.
	If $\alpha$ is infinite and not a limit, then $\gamma(\alpha)$ is the least ordinal so that $\gamma(\alpha)+\alpha>\alpha$.
\end{defn}

\begin{thm}\label{FullScottSent}
	Let $\alpha$ be a computable 
	ordinal and let $\mc{A}$ be a computable structure with a
	$\Pinf{\alpha}$ Scott sentence. Then
	\begin{itemize}
		\item $\mc{A}$ has a $\Pinf{\alpha}$ Scott sentence which is computable from the oracle $\mathbf{0}^{(\gamma(\alpha))}$ (\autoref{37});
		\item $\mc{A}$ has a $\Pinf{2\alpha}$ Scott sentence which is computable (\autoref{55}).
	\end{itemize}
	Moreover, both of these are optimal: For each computable $\alpha\geq 2$, there exists a computable $\mc A$ of Scott complexity $\Pinf{\alpha}$ so that
	\begin{itemize}
		\item $\mc{A}$ has no $\Pinf{\alpha}$ Scott sentence computable in any oracle $\mathbf{0}^{(\delta)}$
		with $\delta<\gamma(\alpha)$ (\autoref{54});
		\item $\mc{A}$ has no $\Sinf{2\alpha}$ Scott sentence which is computable (\autoref{53}).
	\end{itemize}
\end{thm}

Since the least ordinal where $\mc A\not\leq_{\alpha}\mc B$ is always a successor ordinal, we consider $\alpha=\beta+1$ to measure the complexity of the formula on which $\mc A$ and $\mc B$ differ.

\begin{thm}\label{FullBNF}
	Let $\alpha=\beta+1$ be a computable ordinal. Given two computable structures $\mc{A}$ and $\mc{B}$, such that $\mc{A} \nleq_{\alpha} \mc{B}$ (i.e., there is a $\Pinf{\alpha}$ formula $\psi$ with $\mc{A}\models\psi$ and  $\mc{B}\models\neg \psi$)
	\begin{itemize}
		\item there is a $\Pinf{\alpha}$ formula $\varphi$ computable in the oracle $\mathbf{0}^{(\gamma(\alpha))}$
		such that $\mc{A}\models\varphi$ and  $\mc{B}\models\neg\varphi$ (\autoref{bfDef}), and
		\item there is a $\Pinf{2\beta+1}$ formula which is computable
		such that $\mc{A}\models\varphi$ and  $\mc{B}\models\neg \varphi$ (\autoref{512}). 
		\end{itemize}
	Moreover, both of these are optimal: For every computable ordinal $\alpha=\beta+1$, there are computable $\mc{A}$ and $\mc{B}$ so that $\mc A \not \leq_\alpha \mc B$ for which
	\begin{itemize}
		\item $\mc{A}$ and $\mc{B}$ agree on all
		$\Pinf{\alpha}$ formulas computable in an oracle $\mathbf{0}^{(\delta)}$
		with $\delta<\gamma(\alpha)$ (\autoref{sharpGamma}), and
		\item $\mc{A}$ and $\mc{B}$ agree on all computable $\Sinf{2\beta+1}$ formulas (\autoref{513}).
	\end{itemize}
\end{thm}

For finite $\alpha$, these statements recover the corresponding results of Harrison-Trainor \cite{mhtArith}. The main challenge addressed here is extending Harrison-Trainor's theory to arbitrary computable ordinals, where new transfinite phenomena arise and the interaction between syntactic complexity and oracle complexity becomes substantially more delicate.

The results of this paper are driven by a recurring $2\alpha$ versus $\alpha$ phenomenon in computable structure theory. Many structural notions arising from Scott analysis are indexed by an ordinal $\alpha$ and describe behavior at level $\alpha$, yet the effective complexity of recognizing or computing those notions is often on the level of $2\alpha$. The back-and-forth relations provide a prototypical example: although $\leq_\alpha$ measures similarity at level $\alpha$, its definition involves approximately $2\alpha$ quantifier alternations. Similar discrepancies arise when comparing computable infinitary formulas with arbitrary infinitary formulas. As a result, proving sharp Scott-theoretic lower bounds requires a mechanism that can take information available only at complexity level $2\alpha$ and encode it into structures whose distinguishing features appear already at level $\alpha$. The primary technical contribution of this paper is a new method for carrying out precisely this compression of complexity. Although developed to address questions in computable Scott analysis, the resulting machinery is considerably more general and applies whenever one seeks to reflect $2\alpha$-level information in $\alpha$-level structural behavior.

Achieving such compression requires building computable structures whose properties are controlled by information at a much higher level of complexity. We develop such a method in the form of our Belligerent Pairs Theorem and Belligerent Jump Inversion Theorem (\autoref{UnfriendlyPairsInformal} and \autoref{unfriendly jump inv thm informal}).
This effort lies in a long tradition tied to the history of computability theory.
The first use of non-computable information to build an effective object was the introduction of the priority method by Friedburg and Muchnik \cite{Fri57,Muc56} in the 1950s.
They approximated information about the halting set $0'$ to construct recursively enumerable sets that are incomparable in the Turing degrees.
Extensions of the priority method approximating $0''$ or even $0'''$ gave rise to more complicated, but very useful, methods to resolve questions in computability theory \cite{Sho61,Sacks63,Lac76}.
These arguments got tremendously complicated and unwieldy.
Lachlan's $0'''$ argument is even associated with the nickname ``the monster method,'' reflecting both the depth and the opaqueness of his, and related, arguments.
It was clear to many researchers that the development of more and more complicated and intractable methods was unsustainable.
This is evidenced by the emergence of systematic schemas to organize the combinatorics necessary in high-order approximations.
Examples include Harrington's worker constructions \cite{Har76}, Ash's $\eta$-systems \cite{Ash86}, and Montalb\'an's true stages \cite{MonAsh}.
These approaches were useful in the theory of degrees, just as their forebearers were.
That said, they proved especially helpful for computable structure theory.
Particular structural properties (related to the computability of the back-and-forth relations) emerged as important flags, indicating that certain models would be amenable to the use of these approximation schemas.
This crystallized in a series of remarkable victories.
Black-box theorems were developed that can be used with non-computable inputs to retrieve controlled computable output structures without engaging with the complications of the priority method or direct constructions at all.
Notable examples of this include the pair of structures theorem \cite{AK90}, the Ash-Knight metatheorem \cite{AK00}, the method of jump inversion \cite{GHKMMS}, the game metatheorem \cite{MonGame}, and Harrison-Trainor's method of finite unfriendly jump inversion \cite{mhtArith}.

We introduce two primary technical tools in this paper, called the belligerent pairs theorem and the belligerent jump inversion theorem.
As with the previously mentioned theorems, they are black-box tools that can be used to create computable objects from incomputable inputs.
The names of these theorems are meant to evoke the pair of structures theorem \cite{AK90}, the method of jump inversion \cite{GHKMMS} (later isolated in Chapter X.3 of \cite{Part2}), and Harrison-Trainor's method of finite unfriendly jump inversion \cite{mhtArith}.
In broad strokes, our results can be understood as updating the pair of structures theorem and its use in the method of jump inversion, incorporating Harrison-Trainor's unfriendliness.
The pair of structures theorem and (particularly its use in) the method of jump inversion are some of the most widely used tools in computable structure theory.
Jump inversion is a generalization of Marker extensions (from \cite{Mar89}).
Marker extensions allow examples at lower levels of computability to be extended up the arithmetic hierarchy; jump inversion goes beyond and through the whole hyperarithmetic hierarchy.
Harrison-Trainor's method of finite unfriendly jump inversion can be seen as an improvement of Marker extensions from the computability-theoretic perspective. Just as Marker extensions iterate through the arithmetic hierarchy, Harrison-Trainor's method uses unfriendliness to achieve a form of complexity compression unavailable to classical friendly constructions. The two main technical theorems below develop a transfinite analogue of this idea. They are stated informally here and formally in the background section as \autoref{UnfriendlyPairs} and \autoref{unfriendly jump inv thm}.

\begin{thm}[Belligerent Pairs Theorem]\label{UnfriendlyPairsInformal}
   For any computable ordinal $\alpha$, there are two structures $\mc{A}$ and $\mc{B}$ that differ on a $\Pinf{\alpha+1}$ formula, however, for any $\Sigma^0_{2\alpha+1}$ set $S$, there is a uniformly computable sequence of structures so that the $n$th structure in the sequence is isomorphic to $\mc A$ if $n\in S$ and isomorphic to $\mc B$ otherwise.
   Both $\mc{A}$ and $\mc{B}$ are of Scott rank at most $\alpha+1$.
\end{thm}

The significance of this phenomenon is best understood in the context of earlier methods. Traditionally, constructions in computable structure theory exhibit a form of complexity conservation: information supplied to the construction at complexity level $\alpha$ is reflected in the resulting structure only at complexity level $\alpha$ or higher. Put differently, previous coding techniques could faithfully preserve complexity, but they generally did not compress it. The belligerent pairs theorem breaks this pattern. Information at level $2\alpha+1$ determines whether we build $\mc A$ or $\mc B$, yet the outcome is already visible through a $\Pinf{\alpha+1}$ distinction between the resulting structures. Thus information entering the construction at level $2\alpha+1$ is reflected at level $\alpha+1$. Of course, no actual complexity is destroyed: the distinguishing $\Pinf{\alpha+1}$ formula need not be computable, and the missing complexity reappears in the oracle required to describe that formula.
Harrison-Trainor's finite unfriendly jump inversion first exhibited this phenomenon within the arithmetic hierarchy. Extending it to arbitrary computable ordinals is more delicate: the transfinite setting introduces new interactions between syntactic and oracle complexity, reflected in the appearance of the function $\gamma(\alpha)$ in \autoref{FullScottSent} and \autoref{FullBNF}. Our theorem shows that complexity compression persists uniformly throughout the computable ordinals.

\begin{thm}[Belligerent Jump Inversion Theorem]\label{unfriendly jump inv thm informal}
    Given an infinite, computable ordinal $\alpha$ and a $\mathbf0^{(2\alpha+1)}$-computable graph $\mc G$ we can construct a computable graph $\Inv{\alpha}{\mc G}$.
    When compared to $\mc{G}$, the structural properties of $\Inv{\alpha}{\mc G}$ differ by $\alpha$.
    For example, the Scott complexities of $\mc{G}$ and $\Inv{\alpha}{\mc G}$ differ by $\alpha$.
\end{thm}

Viewed from this perspective, belligerent jump inversion should be regarded as a complexity-compression theorem. Earlier jump inversion constructions transformed non-computable information into computable structures while largely preserving the complexity at which that information could be recovered. By contrast, the graph $\Inv{\alpha}{\mc G}$ encodes information about a $\mathbf 0^{(2\alpha+1)}$-computable graph in a form whose structural consequences appear only $\alpha$ levels higher than those of the resulting computable copy. The gap between $2\alpha+1$ and $\alpha$ is precisely what makes the construction useful for the Scott-theoretic applications discussed above.

We can also directly contrast these theorems to the Pair of Structures Theorem \cite{AK00} and the Jump Inversion Theorem  \cite{GHKMMS}. The Pair of Structures Theorem reflects $\Pi_\alpha/\Sigma_\alpha$ information in the realization of a $\Pinf{\alpha}$-formula of a computable structure. To do this, they fix structures $\mc A$ and $\mc B$ which differ on a $\Pinf{\alpha}$ formula, yet are chosen to satisfy a property called ``friendliness'' so as to make them amenable to the method. This friendliness requires that the back-and-forth relations on tuples from $\mc A$ and $\mc B$ are computably enumerable. Our approach, following Harrison-Trainor, is the opposite. We deliberately construct structures whose $\alpha$ back-and-forth relations encode information at complexity level $2\alpha$, and we exploit this complexity directly in proving \autoref{FullScottSent} and \autoref{FullBNF}.
In this sense, Belligerent Jump Inversion is the conceptual opposite of the Pair of Structures Theorem: the latter succeeds because its coding structures are maximally friendly, whereas ours succeed because they are maximally hostile to friendliness.

While our primary applications concern computable Scott analysis, the machinery has further consequences. We mention one representative example, resolving a question of Chen, Gonzalez, and Harrison-Trainor \cite[Question 1.4]{CGHT}.

\begin{restatable}{thm}{ThmMainC}\label{AnswerChenGonzalezHT}
For each computable $\alpha\geq2$, there is a computable structure $\mc{A}$ such that the set
\[\{\mc{B} | \mc{A}\leq_\alpha \mc{B}\} \]
is lightface $\Pi^0_{2\alpha}$, but not lightface $\Sigma^0_{2\alpha}$.
\end{restatable}

The article is written in five sections, including the current one, which includes introductory material.
The second section concerns the background on Scott analysis and jump inversion needed to understand our arguments and their relation to the literature. We also give a more formal statement of our main technical theorems in this section.
The third section demonstrates upper bounds for how difficult it is to distinguish two structures, or one structure from every other structure.
The fourth section develops the machinery of infinite belligerent jump inversions.
The fifth and final section proves a range of applications of this machinery, including the ones mentioned above.

\section{Background and Notation}

We begin by defining fundamental notions needed in the results proved in this paper.
This includes a discussion of infinitary logic, the back-and-forth relations, friendly jump inversion, Scott sentences, Scott rank, Scott complexity, and the universality of graphs. 
We will also provide a few important theorem statements regarding these notions.
For example, we state the pair of structures theorem \cite{AK90} and Scott's isomorphism theorem \cite{Sco65}.
From there, we discuss the results from \cite{mhtArith} regarding finite unfriendly jump inversion, which will inform and, at times, contrast many of our main results.
Finally, we state our two main technical results formally.

\subsection{Infinitary Logic}
The infinitary logic $L_{\omega_1\omega}$ allows countably infinite disjunctions and conjunctions, denoted $\Vvee$ and $\Wwedge$, but only finite strings of quantifiers.
Just like first-order logic, we obtain a normal form for any sentence in $L_{\omega_1\omega}$ by moving negations inside of the quantifiers and infinitary connectives.
In this form, $\Vvee\exists$ alternates with $\Wwedge\forall$.
We classify formulas in this normal form as $\Sigma_\alpha$ or $\Pi_\alpha$ for $\alpha\in\omega_1$:

\begin{enumerate}

\item $\varphi(\bar{x})$ is $\Sinf{0}$ and $\Pinf{0}$ if it is finitary and quantifier-free.

\item  For $\alpha\geq 1$,

\begin{enumerate}

\item  $\varphi(\bar{x})$ is $\Sinf{\alpha}$ if it has form $\Vvee_i\exists \bar{u}_i~\psi_i(\bar{x},\bar{u}_i)$, where each $\psi_i$ is $\Pinf{\beta_i}$ for some $\beta_i < \alpha$,

\item  $\varphi(\bar{x})$ is $\Pinf\alpha$ if it has form $\Wwedge_i\forall\bar{u}_i~\psi_i(\bar{x},\bar{u}_i)$, where each $\psi_i$ is $\Sinf{\beta_i}$ for some $\beta_i < \alpha$.

\item  $\varphi(\bar{x})$ is $\dSinf\alpha$ if it is the conjunction of a $\Sinf{\alpha}$ and $\Pinf\alpha$ formula.

\end{enumerate}
\end{enumerate}

Computable infinitary formulas are $L_{\omega_1\omega}$ formulas in which the infinite disjunctions and conjunctions are over computably enumerable (c.e.) sets.
To make this precise, we need to assign indices to the formulas, as in \cite[Chapter 3]{Part2}.
We classify computable infinitary formulas as $\Sinc{\alpha}$, $\Pinc{\alpha}$, and $\dSinc\alpha$ for computable ordinals $\alpha$ using the same normal form described above with the infinitary connectives restricted to c.e.\ sets. If $\phi\in \Sinf{\alpha}$ and the disjunctions and conjunctions in a formula $\phi$ are $X$-c.e., then we say $\phi\in \Sigma_{\alpha}^X$ (and similar for $\Pinf{\alpha}$ and $\dSinf{\alpha}$ formulas).

We will make use of the following proposition, which states that many formulas that use more than just c.e. disjunctions are still equivalent to computable infinitary formulas.
\begin{prop}[{\cite[Proposition 7.12]{AK00}}]\label{cxAbsorption}
	Suppose that a formula $\varphi(\bar{x})$ is a disjunction of a $\Sigma^0_\alpha$ set of (indices for) $\Sinc\alpha$ formulas. Then $\varphi(\bar{x})$ is equivalent to a $\Sinc\alpha$ formula, and we can find this $\Sinc\alpha$ formula uniformly.
\end{prop}
We will also use the relativized form of the above proposition, which follows readily from the same argument as the one presented in \cite[Proposition 7.12]{AK00}.

\subsection{Back-and-forth relations}

We define the standard (asymmetric) back-and-forth relations $\leq_\alpha$.

\begin{defn}

Let $\mathcal{A},\mathcal{B}$ be structures in the same countable language.  Suppose that $\bar{a}$ in $\mathcal{A}$ and $\bar{b}$ in $\mathcal{B}$ are tuples of the same length.

\begin{enumerate}

\item  $(\mathcal{A},\bar{a})\leq_0(\mathcal{B},\bar{b})$ if $\bar{b}$ and $\bar{a}$ satisfy the same atomic type.\footnote{If the language is infinite, we say $(\mathcal{A},\bar{a})\leq_0(\mathcal{B},\bar{b})$ if the
	 atomic type of $\bar{a}$ and $\bar{b}$ agree on the first $|\bar{a}|$ atomic formulas.}

\item  For $\alpha > 1$, $(\mathcal{A},\bar{a})\leq_\alpha (\mathcal{B},\bar{b})$ if for each $\bar{d}$ and each $1\leq\beta < \alpha$, there exists $\bar{c}$ such that $(\mathcal{B},\bar{b},\bar{d})\leq_\beta(\mathcal{A},\bar{a},\bar{c})$.

\end{enumerate}

\end{defn}

If it is clear that $\bar{a}\in\mathcal{A}$ and $\bar{b}\in\mathcal{B}$, we may write $\bar{a}\leq_\alpha\bar{b}$ instead of $(\mathcal{A},\bar{a})\leq_\alpha (\mathcal{B},\bar{b})$. The expression $\bar{a}\equiv_\alpha\bar{b}$ will also be used as shorthand for $\bar{a}\leq_\alpha\bar{b}$ and $\bar{a}\geq_\alpha\bar{b}$. By carefully writing down the definition of the back-and-forth relation, we note that each step adds two quantifiers. A double transfinite induction gives the following lemma (also seen in the introduction of \cite{CGHT} and Exercise VIII.6 of \cite{Part2}).

\begin{lem}\label{computable version defining back-and-forth by Pi 2 alpha}
	For any $\alpha$ and $\bar{a}\in \mc A$, there is a  $\Pinf{2\alpha}$ formula $\phi(\bar{x})$ so that $\mc B\models \phi(\bar{b})$ if and only if $(\mc A, \bar{a})\leq_\alpha (\mc B, \bar{b})$.
	
	Similarly, for any $\alpha$ and $\bar{a}\in \mc A$, there is a  $\Pinf{2\alpha}$ formula $\rho(\bar{x})$ so that $\mc B\models \rho(\bar{b})$ if and only if $ (\mc B, \bar{b})\leq_\alpha(\mc A, \bar{a})$.
	
	Further, the formulas $\phi$ and $\rho$ are uniformly computable from a presentation of $\alpha$ and $(\mc A, \bar{a})$.
\end{lem}

Karp \cite{Kar65} proved that these back-and-forth relations are intimately tied with the hierarchy of formulas discussed in the previous subsection.

\begin{thm}

Let $\bar{a}$ in $\mathcal{A}$ and $\bar{b}$ in $\mathcal{B}$ be tuples of the same length.  For all countable ordinals $\alpha\geq 1$, the following are equivalent:

\begin{enumerate}

\item  $(\mathcal{A},\bar{a})\leq_\alpha(\mathcal{B},\bar{b})$,

\item  all $\Pinf\alpha$ formulas satisfied by $\bar{a}$ in $\mathcal{A}$ are satisfied by $\bar{b}$ in $\mathcal{B}$,

\item  all $\Sinf \alpha$ formulas satisfied by $\bar{b}$ in $\mathcal{B}$ are satisfied by $\bar{a}$ in $\mathcal{A}$.

\end{enumerate}

\end{thm}

It is often useful to code a piece of non-computable information as the difference between the construction of two distinct computable structures.
One context where this comes up is when we consider elements of $2^\omega$ as structures with domain $\omega$ by bijectively assigning each relation $R$ and tuple $\bar{p}\in\omega^{<\omega}$ of the apropriate size to an index and putting $1$ in that index if $R(\bar{p})$ holds and $0$ if $\lnot R(\bar{p})$ holds.
We say that $(\mathcal{A},\mathcal{B})$ is $(\mathbf\Sigma^0_\alpha,\mathbf\Pi^0_\alpha)$-hard, if for every $\mathbf\Sigma^0_\alpha$ subset $K\subseteq2^\omega$, there is a continuous operator $\Gamma:2^\omega\to2^\omega$ such that $\Gamma^X\cong\mathcal{A}$ if $X\in K$ and $\Gamma^X\cong\mathcal{B}$ if $X\not\in K$.
An immediate consequence of the above theorem is that if $(\mathcal{A},\mathcal{B})$ is $(\mathbf\Sigma^0_\alpha,\mathbf\Pi^0_\alpha)$-hard, then $\mathcal{A}\leq_\alpha\mathcal{B}$.
This is because, were $\mathcal{A}\not\leq_\alpha\mathcal{B}$ to be true, this would be witnessed by some $\Pinf \alpha$ formula $\varphi$ that is true in $\mathcal{A}$ yet false in $\mathcal{B}$. It would follow, if $(\mathcal{A},\mathcal{B})$ were also $(\mathbf\Sigma^0_\alpha,\mathbf\Pi^0_\alpha)$-hard, that membership in an arbitrary $K\in\mathbf\Sigma^0_\alpha$ could be decided in a $\mathbf{\Pi}_\alpha$ manner by checking if $\Gamma(x)\models\varphi$, a contradiction.
It turns out that the converse of this observation is also true.

\begin{thm}[\cite{AK90}]\label{bold-face hardness}
$(\mathcal{A},\mathcal{B})$ is $(\mathbf\Sigma^0_\alpha,\mathbf\Pi^0_\alpha)$-hard if and only if $\mathcal{A}\leq_\alpha\mathcal{B}$.
\end{thm}

What is more, in \cite{AK90}, the tougher direction of the above result is effectivized.
They consider an index sets of recursive structures instead of effective Borel sets.
We say that $(\mathcal{A},\mathcal{B})$ is $(\Sigma^0_\alpha,\Pi^0_\alpha)$-hard if for every $\Sigma^0_\alpha$ set $K\subseteq\omega$, there is an effective map $f:\omega\to\omega$ such that if $x\in K$ then $\Phi_{f(x)}\cong\mathcal{A}$ and  if $x\not\in K$ then $\Phi_{f(x)}\cong\mathcal{B}$
Here, $\Phi_{f(x)}$ codes a structure by coding the corresponding element of $2^\omega$ as a computable real.
To ensure that $(\mathcal{A},\mathcal{B})$ is $(\Sigma^0_\alpha,\Pi^0_\alpha)$-hard in the presence of $\mathcal{A}\leq_\alpha\mathcal{B}$, additional effectiveness assumptions are needed.

\begin{defn}
$\mc{A}$ and $\mc{B}$ are $\alpha$-friendly if the set of triples $(\bar{a},\bar{b},\beta)$ where $\beta<\alpha$, $\bar{a}\in\mc{A}$ and $\bar{b}\in\mc{B}$ so that $\bar{a}\leq_\beta\bar{b}$ is computably enumerable, and similarly for the set of triples $(\bar{b},\bar{a},\beta)$  so that $\bar{b}\leq_\beta\bar{a}$.
\end{defn}

\begin{thm}[The Pair of Structures Theorem \cite{AK90}]\label{lightface hardness}
If $\mathcal{A}\leq_\alpha\mathcal{B}$ and $\mc{A}$ and $\mc{B}$ are $\alpha$-friendly, then $(\mathcal{A},\mathcal{B})$ is $(\Sigma^0_\alpha,\Pi^0_\alpha)$-hard.
\end{thm}

Note that the leftward direction of \autoref{bold-face hardness} follows from \autoref{lightface hardness} by relativizing \autoref{lightface hardness} to a degree which makes $\mc A$ and $\mc B$ friendly.

This remarkable result has been widely applied in computable structure theory \cite{GHKMMS,CGHT,GK}.
In the setting of $\alpha$-friendliness, it is the best you can hope for, as it turns out that  $\mathcal{A}\not\leq_\alpha\mathcal{B}$ is witnessed by a computable formula in this setting.
It turns out that, outside of the setting of $\alpha$-friendliness, more is possible.
In particular, our belligerent pairs theorem will give structures $\mc{A}$ and $\mc{B}$ that have $\mathcal{A}\leq_\alpha\mathcal{B}$  but not $\mathcal{A}\geq_\alpha\mathcal{B}$, which can code roughly double the number of jumps as $\alpha$-friendly structures.
More precisely, we show that there are computable sturctures $\mc{A},\mc{B}$ so that $\mc A \leq_\alpha \mc B$ and $\mathcal{A}$ so that $\mc A \leq_\alpha \mc B$ and $\mathcal{A}\not\geq_\alpha\mathcal{B}$, yet $(\mc A,\mc B)$ is $(\Sigma^0_{2\alpha+1},\Pi^0_{2\alpha+1})$-hard.

The pair of structures theorem is tied to our understanding of jump inversion, just as our new belligerent pairs theorem will be tied to a new notion of belligerent jump inversion.
In \cite{GHKMMS}, they use the pair of structures theorem to prove the iterated-jump inversion theorem, which inverts the notion of the $(\alpha+1)$th ``structural jump'' (see \cite{GHKMMS} for details).
\begin{thm}[\cite{GHKMMS},{\cite[Theorem X.5]{Part2}}]
For every computable ordinal $\alpha$ and every structure $\mc{A}$, there is a structure $\mc{C}$ whose $(\alpha+1)th$ structural jump is effectively bi-interpretable with $\mc{A}\oplus \mathbf{0}^{\alpha+1}$.
\end{thm}
This theorem provides a way to construct structures that resemble $\mc{A}$, but have the information of $\mc{A}$ hidden behind $\alpha+1$ jumps.
It is proven by first considering $\mc{A}$ as an equivalent graph $\mc{G}$.
Then one replaces each edge in $\mc{G}$ with a particular structure and a non-edge with a similar but distinct structure.
The structures used in this process and their theoretical properties are furnished by the pair of structures theorem.
In particular, the structures used are $(\alpha+1)$-friendly.
We will follow a similar pattern of replacing edges and non-edges with particular structures when proving the belligerent equivalent of this theorem, but we will use structures provided by our Belligerent Pairs Theorem.

\subsection{Scott analysis}
One of the main attractions of infinitary logic is its ability to describe any countable structural property, a fact formalized by the following theorem.

\begin{thm}[Vaught's refinement of the Lopez-Escobar Theorem, Lopez-Escobar \cite{Lop65}, Vaught \cite{Vau74}]\label{Lopez-Escobar}
	If $\mathbb{K}\subseteq \Mod(L)$ is any set of $L$-structures which is closed under isomorphism and is $\mathbf{\Pi}^0_\alpha$, then there exists some $\Pinf{\alpha}$ $L$-formula which is true of exactly the models in $\mathbb{K}$. Similarly if $\mathbb{K}$ is $\mathbf{\Sigma}^0_{\alpha}$ then it is defined by a $\Sinf\alpha$ formula.
\end{thm}

We will make use of the following effective version of Theorem \ref{Lopez-Escobar}:

\begin{thm}[Vanden Boom's Lopez-Escobar Theorem \cite{VandenBoom}]\label{vdB}
    If $\mathbb{K}\subseteq \Mod(L)$ is a set of $L$-structures which is closed under isomorphism and is $\Pi^0_\alpha$, then there exists a $\Pinc{\alpha}$ $L$-formula which is true of exactly the models in $\mathbb{K}$. Moreover, one can find the formula effectively in a $\Pi^0_\alpha$-description of $\mathbb{K}$.
\end{thm}

Scott showed that the isomorphism class of any single countable structure is definable in the sense of Theorem \ref{Lopez-Escobar}.
\begin{thm}[\cite{Sco65}]
Any countable sturcture $\mc{A}$ has an $\mc L_{\omega_1,\omega}$ formula $\varphi$ called a \emph{Scott sentence} such that for any countable $\mc{B}$ we have
\[\mc{B}\models \varphi \iff \mc{B}\cong\mc{A}.\]
\end{thm}

Because formulas are ranked in $\mc L_{\omega_1,\omega}$, we can associate a rank to a structure by looking at its Scott sentences.
\begin{defn}
The least $\alpha$ such that $\mc{A}$ has a $\Pinf{\alpha+1}$ Scott sentence is called the \emph{Scott rank} of $\mc{A}$ and is denoted $SR(\mc{A})$.
\end{defn}
There have been several other notions of Scott rank proposed.
This one has emerged as preferred because of its close connection to other concepts throughout mathematical logic.
\begin{thm}[\cite{MonSR}]\label{Robustness}
Fix a countable structure $\mc{A}$. The following are equivalent:
\begin{enumerate}
	\item $SR(\mc{A})\leq\alpha$.
	\item The automorphism orbit of the each tuple in $\mc{A}$ is described by a $\Sinf{\alpha}$ formula.

\end{enumerate}
\end{thm}
There are many other equivalences known to $SR(\mc{A})\leq\alpha$, we only list the most relevant one above.
This theorem, in its many forms, is referred to as Montalb\'an's robustness theorem.
One of the directions of Montalb\'an's robustness theorem, namely (2) implies (1), was known for a long time.
In Scott's original work \cite{Sco65}, he demonstrates how to assemble a Scott sentence from a collection of definitions for automorphism orbits.
\begin{thm}\label{canonicalScott}
If for each $\bar{a}\in\mc{A}$, $\chi_{\bar a}$ defines the automorphism orbit of $\bar a$, then
\[\Wwedge_{\bar{a}\in\mc{A}}\forall\bar{x} \bigg(\chi_{\bar{a}}(\bar{x})\rightarrow \Bigg( \bar{x}\equiv_0\bar{a} \land \bigg(\Wwedge_{a'\in\mc{A}}\exists y \chi_{\bar{a},a'}(\bar{x},y)\bigg) \land \bigg(\forall y\Vvee_{a'\in\mc{A}}\chi_{\bar{a},a'}(\bar{x},y)\bigg)\Bigg)\]
is a Scott sentence for $\mc{A}$.
\end{thm}
\noindent The above sentence is sometimes called a canonical Scott sentence for $\mc{A}$.

One may wonder how well the robustness theorem effectivizes.
It turns out the lightface theory of Scott sentences is not nearly as neat as the boldface theory. The core issue is that the formulas in \autoref{Robustness} needn't be computable infinitary formulas. On the other hand, the formulas in \autoref{computable version defining back-and-forth by Pi 2 alpha} are computable, but require twice as many quantifiers. We will see this tension as a recurring theme.

There are many theorems showing how the lightface theory is wilder, but perhaps the starkest is the following.
\begin{thm}[\cite{ACHT},\cite{KLM}]
The index set of the computable structures with a $\Pinc 2$ Scott sentence is $\Pi_1^1$-complete.
\end{thm}

Many of the applications of belligerent jump inversion in this paper regard the lightface theory of Scott rank and Scott sentences.
Some help to put structure on this difficult part of the theory, others emphasize that some examples are quite misbehaved from this perspective.

There are some already known results about the lightface theory of Scott rank that we will need to prove some of our results.
This is particularly true when it comes to our understanding of Scott complexity, a refinement of Scott rank introduced in \cite{AGNHTT}.
\begin{defn}
The Scott complexity of $\mc{A}$ is the least complexity $\Gamma$ among $\{\Pinf{\alpha},\Sinf{\alpha},\dSinf{\alpha}\}_{\alpha\in\omega_1}$ such that $\mc{A}$ has a Scott sentence of complexity $\Gamma$.
\end{defn}
\noindent In \cite{AGNHTT}, they show that Scott complexity is well defined and that, up to Wadge equivalence, these are the only possible optimal syntactic forms for Scott sentences.
To show that this notion is well-defined, they rely on the following theorem.
\begin{thm}[A. Miller \cite{Mil83}]\label{AMiller}
If $\mc{A}$ has a $\Pinf{\alpha}$ Scott sentence and a $\Sinf{\alpha}$ Scott sentence, it must have a $\dSinf{\beta}$ Scott sentence for some $\beta<\alpha$.
\end{thm}
\noindent A lightface version of this theorem was proven by Alvir, Knight, and McCoy.
\begin{thm}[{Alvir, Knight, and McCoy \cite[Theorem 3.2]{AKM}}]\label{AKMTheorem}
If $\mc{A}$ has a $\Pinc{\alpha}$ Scott sentence and a $\Sinc{\alpha}$ Scott sentence, it must have a $\dSinc{\beta}$ Scott sentence for some $\beta<\alpha$.
\end{thm}

One of the main reasons to care about Scott complexity is that this notion is just as robust as Scott rank, featuring analogs of Montalb\'an's robustness theorem (see \cite{AGNHTT}).
One relevant related fact is that $\Sinf\alpha$ Scott sentences are essentially a parameterized version of $\Pinf\alpha$ Scott sentences, as seen in the following proposition.
\begin{prop}[{\cite[Proposition 1.10]{AGNHTT}}]\label{parameterizeSR}
$\mc{A}$ has a $\Sinf\alpha$ Scott sentence if and only if for some $\bar p\in\mc A$ and $\beta<\alpha$, $(\mc{A},\bar p)$ has a $\Pinf \beta$ Scott sentence.
\end{prop}
Note that, in the special case that $\alpha$ is a limit, a $\Sinf\alpha$ Scott sentence holds if and only if one of its $\Pinf \beta$ disjuncts holds, so the above proposition is straightforward to conclude.
This observation was also made in \cite{AGNHTT}, where they use it to note that $\Sinf\alpha$ and $\dSinf\alpha$ are not possible Scott complexities when $\alpha$ is a limit.

\subsection{Structures as Graphs}
It is often convenient to assume the structures we are working with are graphs, as our jump inversion will be conceived as replacing edges and nonedges with more complicated structures.
This assumption does not lose any generality.
This is because we can transform a structure into a graph without changing any of its relevant properties.
The construction is classical, going back at least as far as Lavrov \cite{Lav63}.
The following is detailed in \cite[Chapter VI.3.2]{Part1}.

\begin{thm}[\cite{Part1} Lemma VI.26 and Theorem VI.27]\label{toGraph}
Every structure $\mc{A}$ has a graph $\mc{G}_\mc{A}$ with which it is effectively bi-interpretable.
In particular (and important for our purposes),
\begin{enumerate}
	\item $\mathbf{d}$ computes a copy of $\mc{A}$ if and only if it computes a copy of $\mc{G}_\mc{A}$.
	\item The Scott complexity (and hence Scott rank) of $\mc{A}$ is the same as that of $\mc{G}_\mc{A}$.
	\item For every $X\in2^\omega$, $\alpha\in\omega_1$ and $\varphi\in\Sigma^X_\alpha$ in the language of $\mc{A}$ there is a $\varphi_*\in\Sigma^X_\alpha$ in the language of graphs with $\mc{A}\models\varphi$ if and only if $\mc{G}_\mc{A}\models\varphi_*$.
	\item For every $X\in2^\omega$, $\alpha\in\omega_1$ and $\psi\in\Sigma^X_\alpha$ in the language of graphs there is a $\psi^*\in\Sigma^X_\alpha$ in the language of $\mc{A}$ with $\mc{A}\models\psi^*$ if and only if $\mc{G}_\mc{A}\models\psi^*$.
\end{enumerate}
The above occurs with all possible uniformity.
\end{thm}

\subsection{Finite Unfriendly Jump Inversion}
Finite unfriendly jump inversion was introduced and explored thoroughly in \cite{mhtArith}.
Most of our results either extend or contrast with the finite case.
Our two main technical results generalize the following finite analogs.

\begin{thm}\label{thm:MHT4.2}
	Fix $n\in\omega$. Let $S \subseteq \mathbb{N}$ be a complete $\Sigma^0_{2n+1}$ set. There are computable structures $\mc{A} \ncong \mc{B}$ and a sequence of computable structures $\mc{C}_i$ such that
	\[ i \in S \Longrightarrow \mc{C}_i \cong \mc{A} \]
	and
	\[ i \notin S \Longrightarrow \mc{C}_i \cong \mc{B}.\]
	Moreover:
	\begin{enumerate}
		\item $\mc A\leq_{n+1}\mc B$, but $\mc{B} \nleq_{n+1} \mc{A}$ and there is a $\Sigma_{n+1}^{\mathbf{0}^{(n)}}$ sentence $\varphi$ such that $\mc{A} \models \varphi$ and $\mc{B} \not\models \varphi$;
		\item $\mc{A}$ and $\mc{B}$ have $\Pinf{n+2}$ Scott sentences. Specifically, $\mc{A}$ has a $\Pinf{n+2}$ Scott sentence and $\mc{B}$ has a $\Pinf{n+1}$ Scott sentence.
	\end{enumerate}
	This is uniform in $n$.
\end{thm}

\begin{thm}\label{thm:MHT4.3}
	For each graph $\mc{G}$ and $n\in\omega$, there is a graph $\Inv{n}{\mc{G}}$ such that:
	\begin{enumerate}
		\item Given a $\mathbf{0}^{(2n)}$-computable copy of $\mc{G}$, there is a computable copy of $\Inv{n}{\mc{G}}$.
		\item Given a $\mathbf{0}^{(n)}$-computable copy of $\Inv{n}{\mc{G}}$ there is a $\mathbf{0}^{(2n)}$-computable copy of $\mc{G}$.
		\item Given a $\Sinf\ell$ sentence $\varphi$, there is a $\Sinf{n+\ell}$ sentence $\varphi_*$ such that for all $\mc{G}$, $\mc{G} \models \varphi$ if and only if $\Inv{n}{\mc{G}} \models \varphi_*$.
		\item For $\ell\geq2$, $\mc{G}$ has a $\Pinf\ell$ Scott sentence if and only if $\Inv{n}{\mc{G}}$ has a $\Pinf{n+\ell}$ Scott sentence.
		\item For all $\ell\geq 3$, $\mc{G}$ has a $\Sinf\ell$ Scott sentence if and only if  $\Inv{n}{\mc{G}}$ has a $\Sinf{n+\ell}$ Scott sentence.
	\end{enumerate}
	Moreover, (1) and (2) are uniform.
\end{thm}
These results are highly related to each other, with \autoref{thm:MHT4.2} used to prove  \autoref{thm:MHT4.3} for a fixed $n$.
From there,  \autoref{thm:MHT4.3} for $n$ is used to prove  \autoref{thm:MHT4.2} for $n+1$.
 \autoref{thm:MHT4.3} describes the procedure that we call unfriendly jump inversion.
That said, direct application of \autoref{thm:MHT4.2} is useful as well.

Many (but not all) of our applications were proven in \cite{mhtArith} at finite levels using this machinery.
We will explicitly note when this is the case throughout our exposition of the applications.

\subsection{Formal Statement of Infinite Belligerent Jump Inversion}
We now give a formal statement of our main technical tools.
They are comparable to the theorems in the finite case, but are extended and more detailed.
We intend these theorems to be usable as a black box in future constructions, so we have taken care to state many of their potentially useful formal properties and relativizations.

\begin{restatable}[Belligerent Pairs Theorem]{thm}{ThmMainA}
	\label{thm:Inf4.2}\label{UnfriendlyPairs}
	Fix an $X$-computable ordinal $\alpha$. Let $S \subseteq \mathbb{N}$ be a complete $\Sigma^0_{2\alpha+1}(X)$ set. There are $X$-computable structures $\mc{A} \ncong \mc{B}$ and a sequence of $X$-computable structures $\mc{C}_i$ such that
	\[ i \in S \Longrightarrow \mc{C}_i \cong \mc{A} \]
	and
	\[ i \notin S \Longrightarrow \mc{C}_i \cong \mc{B}.\]
	Moreover:
	\begin{enumerate}
		\item   $\mc A\leq_{\alpha+1} \mc B$, but $\mc{B} \nleq_{\alpha+1} \mc{A}$; furthermore, there is a  $\Sigma_{\alpha+1}^{X^{(\gamma(\alpha))}}$
		formula $\varphi$ with $\mc{A}\models\varphi$ and  $\mc{B}\not\models\varphi$.
		\item $\mc{A}$ and $\mc{B}$ have $\Pinf{\alpha+2}$ Scott sentences. Specifically, $\mc A$ has a $\Pinf{\alpha+2}$ Scott sentence and $\mc B$ has a $\Pinf{\alpha+1}$ Scott sentence.
	\end{enumerate}
	This is uniform in $\alpha$.
\end{restatable}

\begin{restatable}[Belligerent Jump Inversion Theorem]{thm}{ThmMainB}
	\label{thm:Inf4.3}\label{unfriendly jump inv thm}
	For each graph $\mc{G}$ and $X$-computable
	ordinal $\alpha$, there is a graph $\Inv{\alpha}{\mc{G}}$ such that:
	\begin{enumerate}
		\item Given a $\Delta^0_{2\alpha+1}(X)$ 
		copy of $\mc{G}$, there is an $X$-computable copy of $\Inv{\alpha}{\mc{G}}$. This is uniform in $X$ on indices.
		\item For any $\mathbf{d}\geq_T X^{(\gamma(\alpha))}$, if $\mathbf{d}$ computes a copy of $\Inv{\alpha}{\mc{G}}$, then there is a $\Delta_{\alpha+1}(\mathbf{d})$
		copy of $\mc{G}$. Moreover, if $\mathbf{d}\geq_T X^{(2\alpha+1)}$ (if $\alpha$ is finite, then $\mathbf{d}\geq_T X^{(2\alpha)}$ suffices), then $\mathbf{d}$ computes a copy of $\Inv{\alpha}{\mc{G}}$ if and only if there is a $\Delta_{\alpha+1}(\mathbf{d})$ copy of $\mc{G}$. Both statements are uniform in $X$ on indices.
		\item For any computable ordinal $\beta$, given a $\Sinf\beta$ sentence $\varphi$, there is a $\Sinf{\alpha+\beta}$ sentence $\varphi_*$ such that for all $\mc{G}$, $\mc{G} \models \varphi$ if and only if $\Inv{\alpha}{\mc{G}} \models \varphi_*$. Moreover, if $\phi$ is $\mathbf{a}$-computable, then $\phi_*$ is uniformly computable from $\mathbf{a}\oplus X^{(\gamma(\alpha))}$.
		\item For any computable ordinal $\beta$, given a $\Sinf{\alpha+\beta}$ sentence $\phi$, there exists a $\Sinf{\beta}$ sentence $\phi^*$ so that for all $\mc G$, $\mc G\models \phi^*$ if and only if $\Inv{\alpha}{\mc{G}}\models \phi$. Moreover, if $\phi$ is $\mathbf{a}$-computable, then $\phi^*$ is uniformly computable from $(\mathbf{a}\oplus X^{(2\alpha+1)})^{(\alpha+1)}$ (or $(\mathbf{a}\oplus X^{(2\alpha)})^{(\alpha)}$ if $\alpha$ is finite).
		\item For any ordinal $\beta$, $\mc{G}$ has a $\Pinf\beta$ Scott sentence if and only if $\Inv{\alpha}{\mc{G}}$ has a $\Pinf{\alpha+\beta}$ Scott sentence.
		\item For any ordinal $\beta$, $\mc{G}$ has a $\Sinf\beta$ Scott sentence if and only if $\Inv{\alpha}{\mc{G}}$ has a $\Sinf{\alpha+\beta}$ Scott sentence.
		\item For any ordinal $\beta$, $\mc{G}$ has a $\dSinf \beta$ Scott sentence if and only if $\Inv{\alpha}{\mc{G}}$ has a $\dSinf{\alpha+\beta}$ Scott sentence.
	\end{enumerate}
\end{restatable}

Note that there is a small difference between the above statement and \autoref{thm:MHT4.3}.
In particular, in \autoref{unfriendly jump inv thm}, we look at $\Delta^0_{2\alpha+1}(X)$ copies instead of $\mathbf{0}^{(2n)}$-computable copies. In addition to giving the relativized version of the theorem, we also switch to $\Delta$-notation. This is because Post's theorem differs by one jump on infinite ordinals, i.e., $X^{(2\alpha)}$ is $\Delta^0_{2\alpha}$ instead of being $\Delta^0_{2\alpha+1}$.
This is exactly the cause for the parenthetical comments giving different bounds in (2) and (4). When using jump notation, we have this off-by-one issue.

Note that applying \autoref{unfriendly jump inv thm}(2) to the degree $X^{(\alpha+1)}$, we get a $\Delta_{\alpha\cdot 2 +1}(X)$ copy of $\mc G$, whereas the hypothesis in \autoref{unfriendly jump inv thm}(1) considers $\Delta^0_{2\alpha+1}(X)$ copies of $\mc G$. Thus, we have an asymmetry based on the multiplication by $2$ on different sides of the ordinal $\alpha$.
Of course, in the finite case, it does not matter which side the $2$ is multiplied on.
This creates a symmetry between (1) and (2) in the finite case; in the infinite case, we find an important difference that was not previously observed in the argument.

\section{Computability of Distinguishing Formulas}\label{sec:distinguish}

We aim to understand the oracle needed to witness a $\Pinf\alpha$ difference between two computable structures $\mc A$ and $\mc B$.
This is tied to the issue of friendliness.
An $\alpha$-friendly pair $\mc A$ and $\mc B$ which differ on a $\Pinf{\alpha}$ formula must also differ on a $\Pinc{\alpha}$ formula (follows from the argument in \cite[Theorem 15.2]{AK00}).
Our aim, to use belligerent pairs for jump inversion, is informed by the fact that there is more possible computability-theoretic coding power given by computable structures whose $\Pinf\alpha$ difference is not computable, i.e., is not $\Pinc{\alpha}$.
It is natural, then, to investigate just how much an unfriendly pair may differ from a friendly one by seeing which oracle is required to make an arbitrary computable pair $\mc A$ and $\mc B$ act more like a friendly one. That is, which oracle $X$ has the property that if computable $\mc A$ and $\mc B$ differ on a $\Pinf{\alpha}$ formula, then they differ on a $\Pi_{\alpha}^X$-formula.

This result will also enable us to understand the oracle $X$ needed for a $\Pinf\alpha$ Scott sentence of a computable structure to yield a $\Pi_\alpha^X$ Scott sentence.
We will eventually use the method of belligerent jump inversion to show that this is optimal and contrast this result with a different perspective on the effectiveness of Scott sentences (namely, how many quantifiers are needed for a completely effective Scott sentence).

In the finite setting, Harrison-Trainor showed the following theorems
\begin{thm}[{\cite[Theorem 3.2]{mhtArith}}]
	Let $\mc{A}$ and $\mc{B}$ be computable structures, and suppose that $\mc{A} \not\leq_n \mc{B}$. Then there is a $\Pi_n^{\mathbf{0}^{(n-1)}}$ sentence $\varphi$ such that $\mc{A} \models \varphi$ and $\mc{B} \not\models \varphi$.
\end{thm}

\begin{thm}[{\cite[Theorem 3.5]{mhtArith}}]\label{thm:MHT3.5}
	If a computable structure $\mc{A}$ has a $\Pinf n$ Scott sentence, then it has a $\Pi_n^{\mathbf{0}^{(n)}}$ Scott sentence.
\end{thm}
We complete this story by explaining the transfinite analog, the main results for the section.

\begin{thm}\label{bfDef}
	Let $(\mc{A}_i)_{i \in \omega}$ be a computable list of computable structures. For each infinite computable ordinal $\alpha$, and $\bar{a} \in \mc{A}_i$ there is a  $\Pi_\alpha^{\mathbf{0}^{(\gamma(\alpha))}}$ sentence $\varphi^\alpha_{i,\bar{a}}$ such that for all $\bar{b} \in \mc{A}_j$
	\[ (\mc{A}_i,\bar{a}) \leq_\alpha (\mc{A}_j,\bar{b}) \iff \mc{A}_j \models \varphi_{i,\bar{a}}^\alpha(\bar{b}).\]
	Moreover, this is uniform in $i$, $\bar{a}$, and $\alpha$.
\end{thm}

\begin{proof}
	We begin with the case of a limit ordinal $\alpha$. Then we let $\phi_{i,\bar{a}}^\alpha$ be the $\Pinc{2\alpha}=\Pinc{\alpha}$ formula $\phi$ guaranteed by \autoref{computable version defining back-and-forth by Pi 2 alpha}. Then for all $\bar{b}\in \mc B$ for any structure $\mc B$, $\mc B\models \phi(\bar{b})$ if and only if $(\mc A, \bar{a})\leq_\alpha (\mc B, \bar{b})$.
	
	Now we will verify the result for $\beta+1$, assuming we know it for $\beta=\lambda+n$ where $\lambda$ is a limit ordinal.
	Let $\gamma=\gamma(\beta+1)$, which is also the least such that $\gamma+\beta>\beta$ by definition.
	We proceed similarly to the finite inductive case of the proof of \cite[Theorem 3.4]{mhtArith}.
	For each $i$ and $\bar{c} \in \mc{A}_i$ let $\varphi^\beta_{i,\bar{c}}(\bar{x})$ be a  $\Pi_\beta^{\mathbf{0}^{(\gamma)}}$ formula such that
	\[ (\mc{A}_i,\bar{c}) \leq_\beta (\mc{A}_j,\bar{b}) \iff \mc{A}_j \models \varphi_{i,\bar{c}}^\beta(\bar{b}).\]
	Then for each $i$ and each tuple $\bar{a} \in \mc{A}_i$ consider the $\Pinf{\beta+1}$ formula
	\[ \varphi^{\beta+1}_{i,\bar{a}}(\bar{x}) = \Wwedge_j \Wwedge_{\bar{b},\bar{b}' \in \mc{A}_j} \forall \bar{y} \Vvee_{\bar{a}' \in \mc{A}_i} \begin{cases}
		\neg \varphi^{\beta}_{j,\bar{b}\bar{b}'}(\bar{x},\bar{y})	& (\mc{A}_j,\bar{b}\bar{b}') \not\leq_\beta (\mc{A}_i,\bar{a}\bar{a}') \\
		\top & (\mc{A}_j,\bar{b}\bar{b}') \leq_\beta (\mc{A}_i,\bar{a}\bar{a}')
	\end{cases}.\]
	
	For fixed $\bar{b},\bar{b}' \in \mc{A}_j$ and $\bar{a} \in \mc{A}_i$, deciding for some $\bar{a}' \in \mc{A}_i$ whether $(\mc{A}_j,\bar{b}\bar{b}') \leq_\beta (\mc{A}_i,\bar{a}\bar{a}')$ is $\mathbf{0}^{(2\cdot\beta+1)}$-computable by \autoref{computable version defining back-and-forth by Pi 2 alpha}.
	We use \autoref{cxAbsorption} relativized to $\mathbf{0}^{(\gamma)}$.
	In the subformula
	\[ \Vvee_{\bar{a}' \in \mc{A}_i} \begin{cases}
		\neg \varphi^\beta_{j,\bar{b}\bar{b}'}(\bar{x},\bar{y})	& (\mc{A}_j,\bar{b}\bar{b}') \not\leq_\beta (\mc{A}_i,\bar{a}\bar{a}') \\
		\top & (\mc{A}_j,\bar{b}\bar{b}') \leq_\beta (\mc{A}_i,\bar{a}\bar{a}'),
	\end{cases}\]
	which formula to use is determined by  $\mathbf{0}^{(2\cdot\beta+1)}$.
	
	Note that this is a $\Sigma^0_{\beta}(\mathbf{0}^{(\gamma)})$ set of indices: It suffices to show $\gamma+\beta\ge2\cdot\beta+1$. By assumption the former is greater than $\beta$, and as $\beta$ is infinite, we have $\gamma+\beta \ge\beta+\omega =\lambda+\omega >\lambda+2n+1=2\cdot\beta+1$.
	
	Also, the formulas in the disjunction are $\Sinf\beta$, and, by induction, are uniformly $\mathbf{0}^{(\gamma)}$-computable.
	Therefore, the subformula is equivalent to a  $\Sigma_\beta^{\mathbf{0}^{(\gamma)}}$ formula by the relativized version of \autoref{cxAbsorption}.
	Thus $\varphi^{\beta+1}_{i,\bar{a}}(\bar{x})$ is equivalent to a $\Pi_{\beta+1}^{\mathbf{0}^{(\gamma)}}$ formula.

	We must show that $\varphi^{\beta+1}_{i,\bar{a}}(\bar{x})$ is such that
	\[ (\mc{A}_i,\bar{a}) \leq_{\beta+1} (\mc{A}_j,\bar{b}) \iff \mc{A}_j \models \varphi^{\beta+1}_{i,\bar{a}}(\bar{b}).\]
	This part of the proof is identical to \cite[Theorem 3.4]{mhtArith}.
	
	We first note that $(\mc{A}_i,\bar{a})\models  \varphi^{\beta+1}_{i,\bar{a}}(\bar{a})$.
	Consider $\bar{b},\bar{b}' \in \mc{A}_j$ and $\bar{y}\in \mc{A}_i$.
	Let $\bar{a}'=\bar{y}$.
	If $\bar{b},\bar{b}' \leq_\beta \bar{a},\bar{y}=\bar{a},\bar{a}'$, then the disjunct corresponding to $\bar{a}'$ is true as it has the formula $\top$.
	If $\bar{b},\bar{b}' \not\leq_\beta \bar{a},\bar{y}=\bar{a},\bar{a}'$ then the disjunct corresponding to $\bar{a}'$ is still true as, by definition, $\neg \varphi^\beta_{j,\bar{b}\bar{b}'}(\bar{a},\bar{y})$ holds.
	Therefore, the disjunct corresponding to $\bar{a}'=\bar{y}$ is always true and $(\mc{A}_i,\bar{a})\models  \varphi^{\beta+1}_{i,\bar{a}}(\bar{a})$.
	Because $\varphi^{\beta+1}_{i,\bar{a}}(\bar{x})$ is a $\Pinf{\alpha+1}$ formula, this menas that if $ (\mc{A}_i,\bar{a}) \leq_{\beta+1} (\mc{A}_j,\bar{b})$ we have $\mc{A}_j \models \varphi^{\beta+1}_{i,\bar{a}}(\bar{b})$ as well.
	
	Now suppose that $\mc{A}_j \models \varphi^{\beta+1}_{i,\bar{a}}(\bar{b})$.
	Given, $\bar{b}'\in\mc{A}_j$, we find $\bar{a}'\in\mc{A}_i$ such that $\bar{b}\bar{b}' \leq_\beta \bar{a}\bar{a}'$.
	Note that, taking $\bar{y}=\bar{b}'$, we obtain that
	\[(\mc{A}_j,\bar{b},\bar{b'})\models  \Vvee_{\bar{a}' \in \mc{A}_i} \begin{cases}
		\neg \varphi^{\beta}_{j,\bar{b}\bar{b}'}(\bar{b},\bar{b'})	& (\mc{A}_j,\bar{b}\bar{b}') \not\leq_\beta (\mc{A}_i,\bar{a}\bar{a}') \\
		\top & (\mc{A}_j,\bar{b}\bar{b}') \leq_\beta (\mc{A}_i,\bar{a}\bar{a}')
	\end{cases}.\]
	We always have that $\varphi^{\beta}_{j,\bar{b}\bar{b}'}(\bar{b},\bar{b'})$, so one of the $\top$ conjuncts must hold.
	Thus we may select an $\bar{a}'$ with $(\mc{A}_j,\bar{b}\bar{b}') \leq_\beta (\mc{A}_i,\bar{a}\bar{a}')$.
	We have $(\mc{A}_i,\bar{a}) \leq_{\beta+1} (\mc{A}_j,\bar{b})$ as desired.
\end{proof}

\begin{cor}\label{Cor--IsThisRight?}
	Let $(\mc A_i)_{i\in \omega}$ be a computable list of computable structures. For each infinite computable ordinal $\alpha$, and $\bar{a}\in \mc A_i$ there is a $\Pi^{\mathbf{0}^{(\gamma(\alpha))}}_{\alpha+1}$ sentence $\psi^\alpha_{i,\bar{a}}$ such that for all $\bar{b}\in \mc A_j$,
		\[ (\mc{A}_j,\bar{b}) \leq_\alpha (\mc{A}_i,\bar{a})  \iff \mc{A}_j \models \psi_{i,\bar{a}}^\alpha(\bar{b}).\]
	Moreover, this is uniform in $i$, $\bar{a}$, and $\alpha$.
\end{cor}

\begin{proof}
By the definition of the back-and-forth relations $(\mc{A}_j,\bar{b}) \leq_\alpha (\mc{A}_i,\bar{a})$ if and only if for each $\beta<\alpha$ and $\bar{c}\in \mc{A}_i$ there is a $\bar{d}\in \mc{A}_j$ such that  $(\mc{A}_j,\bar{b},\bar{d}) \geq_\beta (\mc{A}_i,\bar{a},\bar{c})$.
By \autoref{bfDef}, this can be stated as
\[\psi_{i,\bar{a}}^\alpha(\bar{x}):=\Wwedge_{\beta<\alpha}\Wwedge_{\bar{c}\in\mc{A}_i}\exists \bar{d} ~ \varphi^\beta_{\bar{a},\bar{c}}(\bar{x},\bar{d}). \]
The uniformity of \autoref{bfDef} gives that this formula is $\mathbf{0}^{(\gamma(\alpha))}$-computable and it is also $\Pinf{\alpha+1}$ by its apparent syntactic form.
\end{proof}

\begin{cor}\label{cor:twobnfDef}
Let $\alpha$ be an infinite computable ordinal and $\mc{A}$ and $\mc{B}$ be computable structures with $\mc{A}\not\leq_\alpha\mc{B}$.
Then there is a $\Pi_\alpha^{\mathbf{0}^{(\gamma(\alpha))}}$ sentence $\varphi$ with $\mc{A}\models\varphi$ and $\mc{B}\models\neg\varphi$.
\end{cor}

We now turn to showing that $\mathbf 0^{(\gamma(\alpha))}$ is the degree which is required to compute a $\Pinf\alpha$ Scott sentence for a computable structure. We first establish a Lemma about ordinal addition which demonstrates the relative power of computing from $\mathbf 0^{(\gamma(\alpha))}$.

\begin{lem}\label{lem:epsilon ordinal}
	Let $\alpha=\beta+1$ be any infinite successor ordinal. Then there exists some $\epsilon<\beta$ so that $\gamma(\alpha)+\epsilon > 2\alpha$.
\end{lem}
\begin{proof}
	First, suppose that $\beta$ is a successor and $\beta=\epsilon+1$. Then $\gamma(\alpha)+\epsilon + 2\geq \alpha+\omega = 2 \alpha+\omega$. Thus $\gamma(\alpha)+\epsilon > 2\alpha$.
	
	Next suppose that $\beta$ is a limit and $\beta = \sup_i {\delta_i}$. Then $\gamma(\alpha)+\beta\geq \alpha+\omega$, so $\sup_i (\gamma(\alpha)+\delta_i)\geq \alpha+\omega$. Thus there exists some $i$ so that $\gamma(\alpha)+\delta_i>2\alpha$, and we choose $\epsilon = \delta_i$.
\end{proof}

\begin{thm}\label{thm:compScottSent}\label{37}
Let $\alpha$ be a computable ordinal. If a computable structure $\mc{A}$ has a $\Pinf\alpha$ Scott sentence, then it has a $\Pi_\alpha^{\mathbf{0}^{(\gamma(\alpha))}}$ Scott sentence.
\end{thm}

\begin{proof}
Due to \autoref{thm:MHT3.5}, we may assume $\alpha$ is infinite. 
We begin with the case that $\alpha$ is a limit ordinal.
By \autoref{computable version defining back-and-forth by Pi 2 alpha},
there is a $\Pinc{2\alpha}=\Pinc\alpha$ formula $\psi$ such that for any $\mc{B}$, we have $\mc{B}\models\psi$ if and only if $\mc A\leq_\alpha \mc{B}$.
As $\mc{A}$ has a $\Pinf\alpha$ Scott sentence, this means that $\mc{B}\models\psi$ if and only if $\mc{B}\cong\mc{A}$.
Therefore, $\psi$ is a Scott sentence for $\mc{A}$ of the desired form.

In the case that $\alpha$ is not a limit ordinal, we must make a different sort of argument than the one in  \autoref{bfDef}, which only defines the back-and-forth relations in the context of an enumeration of countably many models.
So, we must use the result more indirectly.
Write $\alpha=\beta+1$ and consider $\bar{a}\in\mc{A}$.
As $\mc{A}$ is Scott rank $\beta$, \autoref{Robustness} gives a $\Sinf\beta$ definition $\chi_{\bar{a}}(\bar{x})$ for the automorphism orbit of $\bar{a}$.
We may write
\[ \chi_{\bar{a}}(\bar{x}) = \Vvee_i \exists \bar{y} \rho_i(\bar{x},\bar{y}),\]
with each $\rho_i\in\Pinf{<\beta}$.
One of these disjuncts $\exists \bar{y} \rho_i(\bar{x},\bar{y})$ is true of $\bar{a}$, so that one alone also defines the automorphism orbit of $\bar{a}$.
Say that $\rho_i\in\Pinf\delta$ with $\delta<\beta$ and let
$\bar{a}'$ be the witness for $\exists \bar{y} \rho_i(\bar{a},\bar{y})$. We apply  \autoref{bfDef} to the family $\{\mc A\}$ to get formulas $\phi_{\bar{b}}^\delta$ for $\bar{b}\in \mc A$ and ordinals $\delta\leq \beta$ so that $\bar{b}\leq_\delta \bar{c}$ if and only if $\mc A\models \phi_{\bar{b}}^\delta(\bar{c})$. Then

\[\mc{A}\models \exists \bar{y} \varphi^\delta_{\bar{a},\bar{a'}}(\bar{c},\bar{y}) \Leftrightarrow \mc{A}\models \exists \bar{y} \rho_i(\bar{c},\bar{y})\Leftrightarrow (\mc A,\bar{c})\cong(\mc A,\bar{a}). \]
In particular, the orbit of $\bar{a}\in\mc{A}$ is defined by a $\Sigma_\beta^{\mathbf{0}^{(\gamma(\delta))}}$ formula.
From here on, we continue to use $\chi_{\bar{a}}(\bar{x})$ to denote a formula defining the orbit of $\bar{a}\in\mc{A}$, but additionally assume that it is $\Sigma_\beta^{\mathbf{0}^{(\gamma(\delta))}}$.

We now determine the difficulty of finding a formula $\chi_{\bar{a}}(\bar{x})$. We use the fixed ordinal $\epsilon$ guaranteed by Lemma \ref{lem:epsilon ordinal}. Since each orbit is defined by some $\Sinf{\beta}$-formula, it suffices to find an ordinal $\delta<\beta$ and a tuple $\bar{a}'$ so that $ \phi_{\bar{a},\bar{a'}}^\delta(\bar{c},\bar{d})$ implies
 $\bar{c}\leq_\beta \bar{a}$. Then we can take $\chi_{\bar{a}}(\bar{x})=\exists \bar{y} \phi^\delta_{\bar{a},\bar{a}'}(\bar{x},\bar{y})$.

For any fixed $\bar{a}'$ and $\delta<\beta$ and tuple $\bar{c},\bar{d}$ from $\mc A$, it is $\Sigma^0_\delta(\mathbf{0}^{(\gamma(\alpha))})$ to check if  $\phi_{\bar{a},\bar{a'}}^\delta(\bar{c},\bar{d})$. It is also $\Pi^0_{2\beta}$ to check if $\bar{c}\leq_\beta \bar{a}$ by \autoref{computable version defining back-and-forth by Pi 2 alpha}.
Since $\gamma(\alpha)+\epsilon > 2\beta+1$, it is thus $\Pi^0_\epsilon(\mathbf{0}^{(\gamma(\alpha))})$ to check if  $\bar{c}\leq_\beta \bar{a}$. Thus, to check if
$$\forall \bar{c},\bar{d}\in \mc A \left( \phi_{\bar{a},\bar{a'}}^\delta(\bar{c},\bar{d}) \Rightarrow \bar{c}\leq_\beta \bar{a}\right)$$ is $\Pi^0_{\max(\delta,\epsilon)}(\mathbf{0}^{\gamma(\alpha)})$, thus it is  $\Sigma^0_{\beta}(\mathbf{0}^{\gamma(\alpha)})$ to find one such pair.

We may use orbit definitions to define the canonical Scott sentence for $\mc{A}$, as in \autoref{canonicalScott}.
We have the Scott sentence
\[\Phi = \Wwedge_{\bar{a}\in\mc{A}}\forall\bar{x} \bigg(\chi_{\bar{a}}(\bar{x})\rightarrow \Bigg( \bar{x}\equiv_0\bar{a} \land \bigg(\Wwedge_{a'\in\mc{A}}\exists y \chi_{\bar{a},a'}(\bar{x},y)\bigg) \land \bigg(\forall y\Vvee_{a'\in\mc{A}}\chi_{\bar{a},a'}(\bar{x},y)\bigg)\Bigg). \]
We now analyze the computability of $\Phi$, relative to $\mathbf{0}^{(\gamma(\alpha))}$.
Each of the $\chi_{\bar{a}}$ are $\Sigma^{\mathbf{0}^{(\gamma(\alpha))}}_\beta$, and their indices are uniformly $\Sigma^0_{\beta}(\mathbf{0}^{\gamma(\alpha)})$
 by the above argument.

By \autoref{cxAbsorption}, this gives that each of
$\bar{x}\equiv_0\bar{a}$,   $\bigg(\Wwedge_{a'\in\mc{A}}\exists y \chi_{\bar{a},a'}(\bar{x},y)\bigg)$ and $\bigg(\forall y\Vvee_{a'\in\mc{A}}\chi_{\bar{a},a'}(\bar{x},y)\bigg)$ are $\Pi_{\alpha}^{\mathbf{0}^{(\gamma(\alpha))}}$. Since the outermost conjunction is also $\mathbf{0}^{(\gamma(\alpha))}$-computable, this shows that $\Phi\in \Pi_{\alpha}^{\mathbf{0}^{(\gamma(\alpha))}}$.
\end{proof}

We can extend the above argument to other Scott complexities by using some known facts about lightface Scott sentences. Recall that if $\alpha$ is a limit, then $\Sinf\alpha$ and $\dSinf\alpha$ are not possible Scott complexities \cite{AGNHTT}, so we focus on the successor cases below.

\begin{thm}
Let $\alpha$ be a computable ordinal. If a computable structure $\mc{A}$ has a $\Sinf{\alpha+1}$ Scott sentence, then it has a $\Sigma_{\alpha+1}^{\mathbf{0}^{(\gamma(\alpha))}}$ Scott sentence.
If a computable structure $\mc{A}$ has a $\dSinf{\alpha+1}$ Scott sentence, then it has a $d\mhyph\Sigma_{\alpha+1}^{\mathbf{0}^{(\gamma(\alpha+2))}}$ Scott sentence.
\end{thm}

\begin{proof}
Suppose that $\mc{A}$ has a $\Sinf{\alpha+1}$ Scott sentence.
By \autoref{parameterizeSR}, there is some $\bar{p}$ such that the structure $(\mc{A},\bar{p})$ has a $\Pinf\alpha$ Scott sentence.
By \autoref{thm:compScottSent}, there is a $\Pi_\alpha^{\mathbf{0}^{(\gamma(\alpha))}}$ Scott sentence for $(\mc{A},\bar{p})$ given by $\psi(\bar p)$.
It follows that $\exists \bar x \psi(\bar x)$ is a $\Sigma_{\alpha+1}^{\mathbf{0}^{(\gamma(\alpha))}}$ Scott sentence for $\mc{A}$.

Suppose that $\mc{A}$ has a $\dSinf{\alpha+1}$ Scott sentence.
This means that $\mc{A}$ has both a $\Pinf{\alpha+2}$ and $\Sinf{\alpha+2}$ Scott sentence.
By \autoref{thm:compScottSent} and the above argument, $\mc{A}$ has both a $\Pi_{\alpha+2}^{\mathbf{0}^{(\gamma(\alpha+2))}}$ and $\Sigma_{\alpha+2}^{\mathbf{0}^{(\gamma(\alpha+1))}}$ Scott sentence.
As $\gamma(\alpha+2)\ge\gamma(\alpha+1)$,
 by \autoref{AKMTheorem}, this means that $\mc{A}$ has a $d\mhyph\Sigma_{\alpha+1}^{\mathbf{0}^{(\gamma(\alpha+2))}}$ Scott sentence, as desired.
 \end{proof}

\section{The method of Infinite Belligerent Jump Inversion}
In this section, we prove the two main technical theorems of the paper.
Namely, we prove the Belligerent Pairs Theorem and the Belligerent Jump Inversion Theorem.
We proceed in two steps, first demonstrating much of the desired claim for limit ordinals and using these results to generalize to any countable ordinal.
At limit ordinals $\lambda$, there is the attractive quality that $2\lambda=\lambda$.
In particular, a $\Sigma^0_{2\lambda+1}$ set is just a $\Sigma^0_{\lambda+1}$ set.
To code such a set as the difference of two structures that disagree at the $\lambda+1$ level (as demanded by the Belligerent Pairs Theorem) is therefore achievable even using friendly structures.
Belligerent Jump Inversion at $\lambda$ follows from replacing edges in a graph with the constructed structures that form the desired belligerent pair.
From there, we note that an arbitrary ordinal $\alpha$ can be written as $\lambda+n$ for some $n$.
We then use the construction for $\lambda$ and the finitary constructions for $n$ from \cite{mhtArith} in sequence to construct the structures needed for Belligerent Pairs and Belligerent Jump Inversion.
We then prove that these structure have all of the many qualities we desire.

\subsection{Limit ordinals}

We begin by showing \autoref{thm:Inf4.2} when $\alpha=\lambda$ is a limit ordinal.

\begin{prop}\label{prop:limit4.2}
	Fix an $X$-computable limit ordinal $\lambda$. Let $S \subseteq \mathbb{N}$ be a complete $\Sigma^0_{2\lambda+1}(X)$ set. There are $X$-computable structures $\mc{A} \ncong \mc{B}$ and a sequence of $X$-computable structures $\mc{C}_i$ such that
	\[ i \in S \Longrightarrow \mc{C}_i \cong \mc{A} \]
	and
	\[ i \notin S \Longrightarrow \mc{C}_i \cong \mc{B}.\]
	Moreover:
	\begin{enumerate}
		\item $\mc A\leq_{\lambda+1}\mc B$, but $\mc{B} \nleq_{\lambda+1} \mc{A}$ furthremore, there is a $\Sinc{\lambda+1}$
		formula $\varphi$ with $\mc{A}\models\varphi$ and  $\mc{B}\not\models\varphi$;
		\item $\mc{A}$ and $\mc{B}$ have $\Pinf{\lambda+2}$ Scott sentences. Specifically, $\mc A$ has a d-$\Sinc{\lambda+1}$ Scott sentence and $\mc B$ has a $\Pinc{\lambda+1}$ Scott sentence.
	\end{enumerate}
	This is uniform in $\lambda$.
\end{prop}

\begin{proof}
Consider $\mc{A}=(\omega^\lambda\cdot 2,<)$ and $\mc{B}=(\omega^\lambda,<)$.
If follows from \cite{Ash86} that $\mc{A}\leq_{\lambda+1}\mc{B}$,  $\mc{A}\ngeq_{\lambda+1}\mc{B}$ and that there are $\lambda$-friendly copies relative to $X$ of $\mc{A}$ and $\mc{B}$.
Note that $2\lambda+1=\lambda+1$.
Therefore, $S$ is a complete $\Sigma^0_{\lambda+1}$ set.
Note that, by the Pairs of Structures theorem \cite{AK90}, there is a computable list of strucrtures $\mc{C}_i$ such that
\[ i \in S \Longrightarrow \mc{C}_i \cong \mc{A} \]
	and
\[ i \notin S \Longrightarrow \mc{C}_i \cong \mc{B}.\]

Lastly, note that $\mc{A}$ has a $\dSinc{\lambda+1}$ Scott sentence while $\mc{B}$ has a $\Pinc{\lambda+1}$ Scott sentence (this also follows from \cite{Ash86}).
Therefore, this $\mc{A}$ and $\mc{B}$ are as desired.
\end{proof}

Note that in \autoref{thm:MHT4.2}, for $n\in\omega$, the $\Sinf{n+1}$ sentence true of $\mc{A}$ yet false of $\mc{B}$ is computable in $\mathbf{0}^{(n)}$.
This trend of needing more complexity to describe the formula does not continue uniformly into the transfinite: The oracle required resets to $\mathbf 0$ at limits, and is otherwise a step function with wide plateaus. In the above theorem, the sentence that $\mc{A}$ and $\mc{B}$ differ on is always computable.
Namely, it describes the existence or non-existence of a $\lambda$-limit.
On reflection, this makes sense; for finite $n$, the $2n+1$ jumps that need to be addressed are split between $n+1$ in the alternation of quantifiers in the formula and $n$ in the complexity of calculating said formula.
However, in the limit case, $2\lambda+1=\lambda+1$, so no such split is needed, and all of the jumps can be handled by the alternation of quantifiers in the formula.

We directly use this to prove a weakened form of \TBPT when $\alpha=\lambda$ is a limit ordinal.
Our analysis demonstrates that, at limits, belligerent jump inversion coincides with other, friendly, approaches to jump inversion, such as the one found in \cite{GHKMMS}.
We prove the critical properties of this version of friendly jump inversion listed below.
We focus on these particular properties because inverting at a limit level is a critical piece of the construction of the general Belligerent Pairs Theorem, and we will need these exact properties to ensure that $\Inv{\alpha}{\mc{G}}$ is as claimed.

\begin{prop}\label{prop:limit4.3}
For each graph $\mc{G}$ and $X$-computable limit ordinal $\lambda$, there is a graph $\Inv{\lambda}{\mc{G}}$ such that:
	\begin{enumerate}
		\item For any degree $\mathbf{d}\geq X$, given a $\mathbf{d}^{(2\lambda+1)}$-computable copy of $\mc{G}$, there is a $\mathbf{d}$ computable copy of $\Inv{\lambda}{\mc{G}}$. This is uniform in $X$ on indices.
		\item For any degree $\mathbf{d}\geq X$, given a $\mathbf{d}$ computable copy of $\Inv{\lambda}{\mc{G}}$, there is a $\mathbf{d}^{(\lambda+1)}$ computable copy of $\mc{G}$. This is uniform in $X$ on indices.
		\item For any computable ordinal $\beta$, given a $\Sinf\beta$ sentence $\varphi$, there is a $\Sinf{\lambda+\beta}$ sentence $\varphi_*$ such that for all $\mc{G}$, $\mc{G} \models \varphi$ if and only if $\Inv{\lambda}{\mc{G}} \models \varphi_*$. 
		\item[(5)] For any ordinal $\beta$, if $\mc{G}$ has a $\Pinf\beta$ Scott sentence then $\Inv{\lambda}{\mc{G}}$ has a $\Pinf{\lambda+\beta}$ Scott sentence.
	\end{enumerate}
\end{prop}

\begin{proof}
Given a graph $\mc{G}$ let $\Inv{\lambda}{\mc{G}}$ be a structure over the language $U,R,L$ where $U$ is a unary predicate and $R,L$ are 4-ary predicates defined as follows.
For each $v\in\mc{G}$ there is an $x_v\in \Inv{\lambda}{\mc{G}}$ with the predicate $U$ applied to it.
These are the only elements of $\Inv{\lambda}{\mc{G}}$ with the predicate $U$.
All other elements $a$ will have either $R(u,v,a,a)$ or $L(u,v,a,a)$ with exactly one pair of $U$ points given by $u\neq v$.
We will let
\[R_{u,v}=\{a\in \Inv{\lambda}{\mc{G}} | \Inv{\lambda}{\mc{G}}\models R(u,v,a,a)\},\]
and
\[L_{u,v}=\{a\in \Inv{\lambda}{\mc{G}} | \Inv{\lambda}{\mc{G}}\models L(u,v,a,a)\}.\]
Note that the $R_{u,v}$ and $L_{u,v}$ partition the $\lnot U$ elements.
Let $\mc{A}$ and $\mc{B}$ be the structures from \autoref{prop:limit4.2} at level $\alpha$ realized as graphs (possible by \autoref{toGraph}) with relation $F$.
If $\mc{G}\models E(u,v)$ then there will be an element $a\in R_{u,v}$ for each $a\in\mc{A}$ and an element $c\in L_{u,v}$ for each $c\in\mc{B}.$
Furthermore, we will let $\Inv{\lambda}{\mc{G}}\models R(u,v,a,b)$ if and only if $\mc A\models F(a,b)$ and $\Inv{\lambda}{\mc{G}}\models L(u,v,c,d)$ if and only if $\mc B\models F(c,d)$.
If $\mc{G}\models \lnot E(u,v)$ then there will be an element $a\in A_{u,v}$ for each $a\in\mc{B}$ and an element $c\in L_{u,v}$ for each $c\in\mc{A}.$
Furthermore, we will let $\Inv{\lambda}{\mc{G}}\models R(u,v,a,b)$ if and only if $\mc B\models F(a,b)$ and $\Inv{\lambda}{\mc{G}}\models L(u,v,c,d)$ if and only if $\mc A\models F(c,d)$.
This completes the description of $\Inv{\lambda}{\mc{G}}$.

Note that $2\lambda+1=\lambda+1$.
Therefore, in Item 1, we assume that we have a $\mathbf{d}^{(\lambda+1)}$ computable copy of $\mc{G}$ and construct a $\mathbf{d}$-computable copy of $\Inv{\lambda}{\mc{G}}$.
Note that in $\mc{G}$, assessing if $E(u,v)$ (or indeed $\lnot E(u,v)$) is a uniformly $\mathbf{d}^{(\lambda+1)}$ computable question.
By the pairs of structures theorem \cite{AK90} and the fact that $\mc{A}\equiv_\lambda\mc{B}$  and are $\lambda$-friendly relative to $X$ (as seen in \autoref{prop:limit4.2}) we have $\mathbf{d}\oplus X$ computable procedures that construct $\mc{D}_{u,v}$ and $\mc{C}_{u,v},$ where $\mc{D}_{u,v}\cong\mc{B}$ when $E(u,v)$, $\mc{D}_{u,v}\cong\mc{A}$ when $\lnot E(u,v)$, and $\mc{C}_{u,v}$ exhibits the exact opposite behavior.
We may construct a computable copy of $\Inv{\lambda}{\mc{G}}$ by placing $U$ elements for each vertex in $\mc{G}$ and letting $R_{u,v}\cong \mc{D}_{u,v}$ and $L_{u,v}\cong \mc{C}_{u,v}$.

We now move to Item 2.
Fix a $\mathbf{d}$-computable copy $\mc{D}$ of $\Inv{\lambda}{\mc{G}}$.
There is a $\mathbf{d}$-computable $\Pi_{\lambda+1}$ formula distinguishing $\mc{A}$ from $\mc{B}$.
Looking at two elements $u,v\in U_{\mc{D}}$ it is $\Delta_{\lambda+1}$ in $\mc{D}$ to check if $R_{u,v}\cong\mc{A}$ and $L_{u,v}\cong\mc{B}$ or vice versa.
In total, this means that it is $\Delta_{\lambda+1}$ in $\mathbf{d}$ to check if $u$ and $v$ ought to have an edge between them in $\mc{G}$.
This gives a $\mathbf{d}^{(\lambda+1)}$ copy of $\mc{G}$, as desired.

Consider now Item 3.
Let $\psi$ be the $\Sigma_{\lambda+1}^X$ formula true of $\mc{B}$ but false of $\mc{A}$.
To understand if there is an edge between $u$ and $v$ in $\mc{G}$ given $\Inv{\lambda}{\mc{G}}$, it is $\Delta_{\lambda+1}^X$.
In particular, one checks if $R_{u,v}\models \psi$ or equivalently if $L_{u,v}\models \lnot \psi$.
Translate the sentence $\varphi$ by replacing each instance of $E(u,v)$ with the statement that $R_{u,v}\models \psi$ or $L_{u,v}\models\lnot\psi$, taking care to ensure that the innermost quantifier of $\varphi$ combines with the outermost quantifier of the chosen formula.
It is a straightforward induction to check that the resulting $\psi_*$ is of the desired quantifier complexity and computability theoretic properties and has $\mc{G} \models \varphi$ if and only if $\Inv{\lambda}{\mc{G}} \models \varphi_*$.

We next move to Item 5. First note that the image of the $\Inv{\lambda}{\cdot}$ map is $\Pinf{\lambda+2}$ definable.
The image of $\Inv{\lambda}{\cdot}$ is defined by: \begin{itemize}
	\item a finitary axiom stating that $R_{u,v}$ and $L_{u,v}$ partition the $\lnot U$ elements,
	\item an axiom using the Scott senteces for $\omega^\lambda$ and $\omega^\lambda\cdot2$ to say that $R_{u,v}\cong \omega^\lambda$ or $R_{u,v}\cong \omega^\lambda\cdot2$ for each pair $u,v$,
	\item an axiom using the Scott senteces for $\omega^\lambda$ and $\omega^\lambda\cdot2$ to say that $R_{u,v}\cong \omega^\lambda\cdot 2 \Leftrightarrow L_{u,v}\cong \omega^\lambda$ and $R_{u,v}\cong \omega^\lambda \Leftrightarrow L_{u,v}\cong \omega^\lambda\cdot 2$ for each pair $u,v$.
\end{itemize}
Because $\omega^\lambda\cdot2$ has a $\dSinf{\lambda+1}$ Scott sentence while $\omega^\lambda$ has a $\Pinf{\lambda+1}$ Scott sentence, the above is $\Pinf{\lambda+2}$.
Lastly, by construction, this defines the image of $\Inv{\lambda}{\cdot}$, as desired.

Note Item 3 extends to $\Pinf{\lambda+\beta}$ formulas by taking negations.
Now, if $\chi$ is a Scott sentence for $\mathcal{G}$ and $\nu$ defines the image of $\Inv{\lambda}{\cdot}$, then $\chi_*\land \nu$ is a Scott sentence for $\Inv{\lambda}{\mc{G}}$ of the desired complexity. 
\end{proof}

\subsection{All countable ordinals}
In this section, we establish the general case of the Belligerent Pairs Theorem and the Belligerent Jump Inversion Theorem. We will make use of the following Theorem in the proof of the Belligerent Pairs Theorem:

\begin{thm}\label{512}
	Fix a computable ordinal $\alpha$ and computable structures $\mc A$ and $\mc B$ so that $\mc A \not \leq_{\alpha+1}\mc B$. Then there exists a $\Pinc{2\alpha+1}$-sentence $\phi$ so that $\mc A\models \phi$ and $\mc B\models \neg\phi$.
\end{thm}
\begin{proof}
	Since $\mc A \not \leq_{\alpha+1}\mc B$ there exists a tuple $\bar{b}\in \mc B$ so that there is no tuple $\bar{a}\in \mc A$ so that $(\mc A,\bar{a})\leq_{\alpha}(\mc B,\bar{b})$. Fix the formula $\rho(\bar{x})$ guaranteed by Lemma \ref{computable version defining back-and-forth by Pi 2 alpha} so that $\mc C\models \rho(\bar{c})$ if and only if $(\mc C,\bar{c})\leq_\alpha (\mc B,\bar{b})$. Thus $\mc A \models \forall \bar{x} \neg \rho(\bar{x})$ whereas $\mc B \models \exists \bar{x}\rho(\bar{x})$, showing that there is a $\Pinc{2\alpha+1}$-formula true in $\mc A$ but not in $\mc B$.
\end{proof}

Note that Lemma \ref{computable version defining back-and-forth by Pi 2 alpha} shows that whenever $\mc A$ is computable and $\mc{A}\not\leq_{\alpha+1}\mc{B}$, then there is a $\Pinc{2\alpha+2}$-formula $\psi$ so that $\mc A\models \psi$ and $\mc B\models \neg \psi$, but \autoref{512} gives a slight improvement of this.

\ThmMainA*
\begin{proof}
We assume that $\alpha$ is computable, i.e., $X=\emptyset$. The proof relativizes for any $X$.
Let $\alpha=\lambda+n$ where $\lambda$ is a limit ordinal and $n\in\omega$.
If $n=0$, then the result follows at once from  \autoref{prop:limit4.2}.
If $\lambda=0$, then the result follows at once from \autoref{thm:MHT4.2}.
Note that $2\alpha+1=2(\lambda+n)+1=\lambda+2n+1$, so $S$ is a complete $\Sigma^0_{\lambda+2n+1}=\Sigma^0_{2n+1}(\mathbf{0}^{(\lambda+1)})$ set.
Use  \autoref{thm:MHT4.2} relativized to $\mathbf{0}^{(\lambda+1)}$ to obtain the described $\mathbf{0}^{(\lambda+1)}$-computable structures $\widehat{\mc{A}},\widehat{\mc{B}},\widehat{\mc{C}}_i$.
Uniformly apply \autoref{prop:limit4.3} to obtain $\mc A = \Inv{\lambda}{\widehat{\mc{A}}},\mc{B} = \Inv{\lambda}{\widehat{\mc{B}}},\mc{C}_i = \Inv{\lambda}{\widehat{\mc{C}}_i}$, a computable sequence of structures.
We claim that this sequence of structures acts as we desire.
Note that $\Inv{\lambda}{\cdot}$ is injective up to isomorphism as it pushes forward Scott sentences.
Therefore,
	\[ i \in S \Longrightarrow 
	\mc{C}_i \cong \mc{A}\]
	and
	\[ i \notin S \Longrightarrow
	\mc{C}_i \cong \mc{B}. 
	\]

Next, let $\psi$ be the $\Sinf{n+1}$ sentence true of $\widehat{\mc{A}}$ yet false of $\widehat{\mc{B}}$.
Note that $\mc A = \Inv{\lambda}{\widehat{\mc{A}}}\models\psi_*$ and $\mc B = \Inv{\lambda}{\widehat{\mc{B}}}\models\lnot\psi_*$.
Therefore, $\mc{B} \nleq_{\lambda+n+1} \mc{A}$, and thus $\mc{B} \nleq_{\alpha+1} \mc{A}$. By  \autoref{bfDef}, there is some $\Pi_{\alpha+1}^{\mathbf{0}^{(\gamma(\alpha))}}$-sentence  $\phi$ so that $\mathcal{A}\models \phi$ and $\mathcal{B}\models\neg\phi$.
Note that \autoref{512} shows that if it were the case that $\mc A\not\leq_{\alpha+1} \mc B$, then we would have a $\Pinc{2\alpha+1}$-sentence $\rho$ so that $\mc A\models \rho$ and $\mc B\models \neg\rho$. But then the set of $i$ so that $\mc C_i\models \rho$ would give a $\Pi^0_{2\alpha+1}(X)$ description of $S$, contradicting that $S$ is $\Sigma^0_{2\alpha+1}(X)$-complete. Thus  $\mc A\leq_{\alpha+1} \mc B$.

Lastly, as $\widehat{\mc{A}}$ and $\widehat{\mc{B}}$ have $\Pinf{n+2}$ Scott sentences, it follows from \autoref{prop:limit4.3}(5) that $\mc A = \Inv{\lambda}{\widehat{\mc{A}}}$ and $\mc B = \Inv{\lambda}{\widehat{\mc{B}}}$ have a $\Pinf{\lambda+n+2}=\Pinf{\alpha+2}$ Scott sentence and a $\Pinf{\lambda+n+1}=\Pinf{\alpha+1}$ Scott sentence, respectively, as desired.
\end{proof}

\ThmMainB*
\begin{proof}
\autoref{toGraph} allows us to transform any structure into a graph without changing any of the relevant information.
Therefore, we may abuse notation and uniformly build a structure over any language and call it $\Inv{\alpha}{\mc{G}}$.
This proof will mirror the approach spelled out in \autoref{prop:limit4.3}.

Given a graph $\mc{G}$, let $\Inv{\alpha}{\mc{G}}$ be a structure over the language $U,R,L$ where $U$ is a unary predicate and $R,L$ are 4-ary predicates defined as follows.
For each $v\in\mc{G}$ there is an $x_v\in \Inv{\alpha}{\mc{G}}$ with the prediacte $U$ applied to it.
These are the only elements of $\Inv{\alpha}{\mc{G}}$ satisfying the predicate $U$.
All other elements $a$ will have either $R(u,v,a,a)$ or $L(u,v,a,a)$ with exactly one pair of $U$ points given by $u\neq v$.
We will let
\[R_{u,v}=\{a\in \Inv{\alpha}{\mc{G}} | \Inv{\alpha}{\mc{G}}\models R(u,v,a,a)\},\]
and
\[L_{u,v}=\{a\in \Inv{\alpha}{\mc{G}} | \Inv{\alpha}{\mc{G}}\models L(u,v,a,a)\}.\]
Note that the $R_{u,v}$ and $L_{u,v}$ partition the $\lnot U$ elements.
Let $\mc{A}$ and $\mc{B}$ be the structures realized as graphs with relation $F$ from \autoref{thm:Inf4.2} applied at level $\alpha$.
If $\mc{G}\models E(u,v)$ then there will be an element $a\in R_{u,v}$ for each $a\in\mc{A}$ and an element $c\in L_{u,v}$ for each $c\in\mc{B}.$
Furthermore, we will let $\Inv{\alpha}{\mc{G}}\models R(u,v,a,b)$ if and only if $\mc A\models F(a,b)$ and $\Inv{\alpha}{\mc{G}}\models L(u,v,c,d)$ if and only if $\mc B\models F(c,d)$.
If $\mc{G}\models \lnot E(u,v)$ then there will be an element $a\in A_{u,v}$ for each $a\in\mc{B}$ and an element $c\in L_{u,v}$ for each $c\in\mc{A}.$
Furthermore, we will let $\Inv{\alpha}{\mc{G}}\models R(u,v,a,b)$ if and only if $\mc B\models F(a,b)$ and $\Inv{\alpha}{\mc{G}}\models L(u,v,c,d)$ if and only if $\mc A\models F(c,d)$.
This completes the description of $\Inv{\alpha}{\mc{G}}$.

Note that in a $\Delta^0_{2\alpha+1}(X)$
copy of $\mc{G}$, assessing if $E(u,v)$ (or indeed $\lnot E(u,v)$) is a uniformly
$\Delta^0_{2\alpha+1}(X)$
question.
By \autoref{thm:Inf4.2}, we have computable procedures that construct $\mc{D}_{u,v}$ and $\mc{C}_{u,v},$ where $\mc{D}_{u,v}\cong\mc{B}$ when $E(u,v)$ and $\mc{D}_{u,v}\cong\mc{A}$ when $\neg E(u,v)$, and $\mc{C}_{u,v}$ exhibits the exact opposite behavior.
We may construct a computable copy of $\Inv{\alpha}{\mc{G}}$ by placing $U$ elements for each vertex in $\mc{G}$ and letting $R_{u,v}\cong \mc{D}_{u,v}$ and $L_{u,v}\cong \mc{C}_{u,v}$.

We now move to Item 2.
Suppose that $\mathbf{d}\geq_T X^{(\gamma(\alpha))}$ computes a copy $\mc{D}$ of $\Inv{\alpha}{\mc{G}}$. There is a  $\Pi_{\alpha+1}^{X^{(\gamma(\alpha))}}$ formula $\psi$ distinguishing $\mc A$ from $\mc B$. Thus, since
$\mathbf{d}$ can compute the formula $\psi$ and also the atomic diagram of $\mc D$, looking at two elements $u,v\in U_{\mc{D}}$ it is $\Delta_{\alpha+1}(\mathbf{d})$ to check if $R_{u,v}\cong\mc{A}$ and $L_{u,v}\cong\mc{B}$ or vice versa. Thus there is a $\Delta_{\alpha+1}(\mathbf{d})$ copy of $\mc G$. Suppose that $\mathbf{d}\geq X^{(2\alpha+1)}$ (or $\mathbf{d}\geq X^{(2\alpha)}$ if $\alpha$ is finite). It follows from \autoref{computable version defining back-and-forth by Pi 2 alpha} that $\mc A$ and $\mc B$ are $\alpha+1$-friendly relative to $\mathbf{d}$. By \autoref{lightface hardness}, if there is a $\Delta_{\alpha+1}(\mathbf{d})$ copy of $\mc G$, then $\mathbf{d}$ computes a copy of $\Inv{\alpha}{\mc{G}}$.

We now move to Item 3.
Let $\psi$ be the $\Sinf{\alpha+1}$ formula true of $\mc{A}$ but false of $\mc{B}$. Note that it is $\Delta^{\text{in}}_{\alpha+1}$ (i.e., both $\Sinf{\alpha+1}$ and also $\Pinf{\alpha+1}$ to say that $R_{u,v}\models \psi$, since $R_{u,v}\models \psi\Leftrightarrow L_{u,v}\models \neg \psi$).
We translate the sentence $\varphi$ by replacing each instance of $E(u,v)$ with the $\Delta^{\text{in}}_{\alpha+1}$ formula $R_{u,v}\models \psi$ (formally, we use either $R_{u,v}\models \psi$ or $L_{u,v}\models \neg \psi$ depending on whether $E$ appears positively or negatively and depending on the innermost quantifier in $\phi$, though it is common to simply think of such formulae as $\Delta^{\text{in}}_{\alpha+1}$ and thus combining with whichever quantifier comes next).
The resulting $\phi_*$ is of the desired complexity and has $\mc{G} \models \varphi$ if and only if $\Inv{\alpha}{\mc{G}} \models \varphi_*$. Note that this gives $\phi_*$ uniformly computable from $\mathbf{a}\oplus X^{(\gamma(\alpha))}$ if $\mathbf{a}$ computes $\phi$, since $X^{(\gamma(\alpha))}$ computes the formula $\psi$.

Next, we consider Item 4. We are given a $\Sinf{\alpha+\beta}$ formula $\phi$. Let $\mathbf d$ be large enough that that $\mc A$ and $\mc B$ are $(\alpha+1)$-friendly relative to $\mathbf{d}$ (i.e., $\mathbf{d}\geq X^{(2\alpha+1)}$ suffices for this) and that $\phi$ is $\mathbf{d}$-computable. For any structure $\mc C$, uniformly in $\mc C \oplus \mathbf{d}^{(\alpha+1)}$, we can find a set $Y$ so that $Y\geq \mathbf{d}$ and $Y^{(\alpha+1)}\equiv_T \mc C \oplus \mathbf{d}^{(\alpha+1)}$ \cite{Macintyre1977}. By \autoref{lightface hardness}, $Y$ computes a copy of $\Inv{\alpha}{\mc{C}}$. Thus uniformly in a $\Sigma_{\alpha+\beta}(Y)\subseteq \Sigma_\beta(Y^{(\alpha+1)}) = \Sigma_\beta(\mc C\oplus \mathbf{d}^{(\alpha+1)})$ way, we can tell if $\Inv{\alpha}{\mc{C}}\models \phi$. By \autoref{vdB}, there is a formula $\phi^*$ which is $\Sigma^{\mathbf{d}^{\alpha+1}}_\beta$ defining the set of $\mc C$ so that $\Inv{\alpha}{\mc{C}}\models \phi$, and by the uniformity of Vanden Boom's theorem, we can uniformly find the $\mathbf{d}^{(\alpha+1)}$-computation for $\phi^*$. Note that if $\mathbf{a}$ computes $\phi$, then $\mathbf{d}=\mathbf{a}\oplus X^{(2\alpha+1)}$ suffices.
Note that if $\alpha$ is finite, then $\mathbf{d}\geq X^{(2\alpha)}$ suffices to make $\mc A $ and $\mc B$ be $(\alpha+1)$-friendly relative to $\mathbf{d}$. Also, we have
 $\Sigma_{\alpha+\beta}(Y)=\Sigma_\beta(Y^{(\alpha)})$, so the above argument works to give a $\mathbf{d}^{(\alpha)}$-computation of $\phi^*$ with $\mathbf{d}=\mathbf{a}\oplus X^{(2\alpha)}$.

Note that (3) extends to $\Pinf{\alpha+\beta}$ by taking negations and to $\dSinf{\alpha+\beta}$ formulas by taking conjunctions. We now prove the forward direction of (5), (6), and (7).
This follows directly from our extended (3) and the fact that the image of the $\Inv{\alpha}{\cdot}$ map is $\Pinf{\alpha+2}$ definable.
In particular, if $\chi$ is a Scott sentence for $\mathcal{A}$ and $\nu$ defines the image of $\Inv{\alpha}{\cdot}$, then $\chi_*\land \nu$ is a Scott sentence for $\Inv{\alpha}{\mc{G}}$.
The image of $\Inv{\alpha}{\cdot}$ is defined by
\begin{itemize}
	\item a finitary axiom stating that $R_{u,v}$ and $L_{u,v}$ partition the $\lnot U$ elements
	\item an axiom using the Scott senteces for $\mc{A}$ and $\mc{B}$ to say that $R_{u,v}\cong \mc{A}$ or $R_{u,v}\cong \mc{B}$
	\item an axiom using the Scott senteces for $\mc{A}$ and $\mc{B}$ to say that $L_{u,v}\cong \mc{A}$ or $L_{u,v}\cong \mc{B}$
	\item an axiom stating that $R_{u,v}\models\psi\implies L_{u,v}\models\lnot\psi$ and $R_{u,v}\models\lnot\psi\implies L_{u,v}\cong \lnot\psi$ where $\psi$ is the $\Sinf{\alpha+1}$ sentence so that $\mc B\models \psi$ and $\mc A\models \neg \psi$.
\end{itemize}
Because $\mc{A}$ and $\mc{B}$ have $\Pinf{\alpha+2}$ Scott sentences and $\psi$ is $\Sinf{\alpha+1}$, the above is $\Pinf{\alpha+2}$.
Lastly, by construction, this defines the image of $\Inv{\alpha}{\cdot}$, as desired.

Note also that (4) extends to $\Pinf{\alpha+\beta}$ by taking negations and to $\dSinf{\alpha+\beta}$ formulas by taking conjunctions.
Now, suppose that $\chi$ is a Scott sentence for $\Inv{\alpha}{\mc G}$. Then for any structure $\mc M$, $\mc M\models \chi^*$ implies that $\Inv{\alpha}{\mc M}\models \chi$, thus $\Inv{\alpha}{\mc M}\cong \Inv{\alpha}{\mc G}$, thus $\mc M \cong \mc G$. So $\chi^*$ is a Scott sentence for $\mc G$. This proves the leftward direction of (5), (6), and (7).
\end{proof}

Having these results, we shift to applications of these tools, which will demonstrate both their applicability and also their sharpness.

\section{Applications}

This section contains various applications of the previous section and demonstrates that the results in the previous section are the best possible.
The applications in this section focus on several goals.
The first is to establish our results on computable Scott analysis as in \autoref{FullScottSent}. The second is to establish our results on the complexity of the back-and-forth relations as in \autoref{FullBNF}. From there, we give applications to answering a question of Chen, Gonzalez, and Harrison-Trainor as in \autoref{AnswerChenGonzalezHT} and also proving a result related to recent work of Gonzalez and Knight on the complexity of the collection of structures of Scott rank $\leq \alpha$. Finally, using the fact that our applications are sharp, we conclude that the Belligerent Pair Theorem and the Belligerent Jump Inversion Theorem are sharp: They cannot be improved since the bounds in the applications cannot be improved.

Before demonstrating applications of infinite belligerent jump inversion, we give a general outline of how the applications generally proceed, which we hope will be useful to anyone trying to apply this theorem.
The general form of a problem where belligerent jump inversion may be helpful is when you desire to prove completeness at the level $2\alpha+k$ for some property that scales with $\alpha$.
In other terms, belligerent jump inversion can add two quantifiers of coding at the cost of only sacrificing one level in $\alpha$.
This differs from the more typical ``friendly'' jump inversion (e.g., the one in \cite{GHKMMS}) that only adds one quantifier per level of $\alpha$.
Problem set-ups of this form are typical when moving between effective and non-effective versions of a particular concept, or, relatedly, when discussing back-and-forth relations.

Once the appropriate setting is identified, one way to apply infinite belligerent jump inversion is the following three-step method.
First, some sort of base case $\alpha=m$ is demonstrated using direct coding arguments.
Second, belligerent jump inversion is applied to a relativized form of the base case to obtain the desired result for all ordinals of the form $\beta+m$.
Lastly, edge cases of the form $\lambda+p$ for $p<m$ and $\lambda$ a limit ordinal are considered to fill in the gaps left by the second step of the proof.

We demonstrate this method by proving an extension of \cite[Theorem 6.1]{mhtArith}.
\begin{thm}\label{thm:lightfaceScott SentenceHard}
Let $\alpha\geq2$ be a computable ordinal and let $S$ be a complete $\Pi^0_{2\alpha}$ set. There is a computable structure $\mc{A}_\alpha$ with a $\Pinf\alpha$ Scott sentence and a uniformly computable sequence of structures $\mc{C}_{i,\alpha}$ such that
\[i\in S 	\Longleftrightarrow \mc{C}_{i,\alpha}\cong\mc{A}_\alpha\]
\end{thm}

\begin{proof}
The base case is given by $\alpha=2$.
The proof of this fact is given by direct coding and provided in \cite{ACHT}.
The proof readily relativizes.

Let $\alpha=\beta+2$ for some ordinal $\beta$.
Relativize the above construction to $T=\mathbf{0}^{(2\beta+1)}$ to obtain a $\mc{A}_2(T)$ and $\mc{C}_{i,2}(T)$ uniformly computable in $\mathbf{0}^{(2\beta+1)}$ so $i\in S$ if and only if $\mc{C}_{i,2}(T)\cong \mc A_2(T)$. Note that $S$ is a $\Pi^0_4(\mathbf{0}^{(2\beta+1)})=\Pi^0_{2\beta+4}=\Pi^0_{2\alpha}$ set.
We now let
\[\mc{A}_\alpha :=  \Inv{\beta}{\mc{A}_2(T)} \text{   and   } \mc{C}_{i,\alpha} :=  \Inv{\beta}{\mc{C}_{i,2}(T)}.\]
By \autoref{thm:Inf4.3} Part 1, $\mc{A}_\alpha$ and $\mc{C}_{i,\alpha}$ are uniformly computable.
Note that $\Inv{\beta}{\cdot}$ is injective up to isomorphism as it pushes forward Scott sentences by \autoref{thm:Inf4.3} Part 3.
In particular, we maintain that
\[i\in S 	\Longleftrightarrow \mc{C}_{i,\alpha}\cong\mc{A}_\alpha\]
as desired. It remains to consider the cases that $\alpha=\lambda$ or $\alpha=\lambda+1$ where $\lambda$ is a limit.

We now consider the case that $\alpha=\lambda$ is a limit ordinal.
Let $\delta_i$ be a computable sequence of ordinals so that $\lambda=\sup_i \delta_i$.
Let $\mc{A}$ be the many sorted structure with $\omega^{\delta_i}$ on sort $i$.
Note that $\mc{A}$ has a $\Pinf\lambda$ Scott sentence.
Namely, the conjunction of the Scott sentences for each $\omega^{\delta_i}$ restricted to sort $i$ is such a sentence.

Let $S$ be a complete $\Pi^0_{2\lambda}=\Pi^0_{\lambda}$ set.
Write $S=\bigcap_{j\in\omega} U_j$ where $U_j$ is $\Pi^0_{\delta_j}$.
Let $\mc{M}_{j,i}$ be a computable structure that is isomorphic to $\omega^{\delta_j}$ if $i\in U_j$ and $\omega^\lambda$ if $i\not\in U_j$.
This is possible by the pair of structures theorem \cite{AK90}.
We let $\mc{C}_i$ be a many sorted structure with the $j^{th}$ sort isomorphic to $\mc{M}_{j,i}$.
If $i\in S$, it is in each $U_j$, so $\mc{M}_{j,i}\cong \omega^{\delta_j}$.
Thus, $\mc{C}_i\cong \mc{A}$.
If $i\not\in S$, there is some $i\not\in U_j$, so $\mc{M}_{j,i}\cong \omega^{\lambda}$.
As $\mc{A}$ has no sort with structure isomorphic to $\omega^{\lambda},$ we obtain that  $\mc{C}_i\not\cong \mc{A}$.
Therefore, the $\mc{A}$ and $\mc{C}_i$ are as desired.

We lastly focus on the case that $\alpha=\lambda+1$ for a limit ordinal $\lambda$. Let $\delta_i$ be a computable sequence of ordinals so that $\lambda=\sup_i \delta_i$ (without loss of generality, say $\delta_i\ge i$).
We first note that it suffices to construct $\mc A_\alpha$ when $S$ is a complete $\Sigma^0_{\lambda+1}$ set.
This is because given a $\Pi^0_{\lambda+2}$ set $S'$ we may write $S'$ as an intersection of countably many $\Sigma^0_{\lambda+1}$ sets $S'_i$.
Given the procedure for $S$, for each $i$, we may construct $\mc A_\alpha$ when $x\in S'_i$.
By adding sorts, we may construct the structure that has $\mc A_\alpha$ on each sort exactly when $x\in S'_i$.
This many-sorted structure still has a $\Pinf\alpha$ Scott sentence, giving the desired result.
So, we assume that for all $x$, $x\in S\iff\exists y\;(\langle x,y\rangle\in\bigcap_{j\in\omega} U_j)$ where $U_j$ is $\Pi_{\delta_j}$.

We build a tree $T\subseteq 2^{<\omega}$ encoding membership questions about $S$: We say that $j$ is $(x,y)$-denying if $j$ is least so that $\langle x,y\rangle\notin U_j$. We say that $j$ is $x$-negating if there is some $y$ so that $\langle x,y_0\rangle \in U_j$ for every $y_0<y$ and $j$ is $(x,y)$-denying. Note that $x\in S$ if and only if $\{j\mid j \text{ is $x$-negating}\}$ is finite.  We now define a tree $T_x$ so that $T_x$ has a non-isolated path if and only if $x\notin S$. We describe $T_x$ as a union of trees $T_x[s]$ which are $T_x$ restricted to height $s$. At each stage $s$, we have a particular node $\sigma^x_s\in T_x[s]$ which is `active'. We begin at stage $0$ with $T_x[0]=\{\lambda\}$ and $\sigma^x_s=\lambda$. If $s+1$ is not $x$-negating, then $T_x[s+1]=T_x[s]\cup \{\rho^\smallfrown 0 \mid \rho\in T_s[x]\}$, and $\sigma^x_{s+1}=\sigma^x_s\smallfrown 0$.  If $s+1$ is $x$-negating, then we let $T_x[s+1]=T_x[s]\cup \{\rho\smallfrown 0 \mid \rho\in T_x[s]\}\cup \{\sigma^x_s\smallfrown 1\}$ and we let $\sigma^x_{s+1}=\sigma^x_s\smallfrown 1$. We let $\sigma^x=\bigcup_{s}\sigma^x_s$.
Finally, we let $T$ be the prefix-closure of $\bigcup_{x\in \omega} \{0^x\smallfrown 1 \smallfrown \rho \mid \rho\in T_x\}$.

Let $\Th$ be the theory in the language $\mathcal L = \{U_i \mid i\in \omega\}$ with the following axioms, where $U_i^{0}:=\neg U_i$ and $U_i^1:=U_i$.
\begin{itemize}
	\item For every $n\in \omega$ and $\sigma\in T$: $\exists^n y \bigwedge_{m<\vert \sigma \vert} U^{\sigma(i)}_{i}(y)$
	\item For every $\sigma\notin T$: $\neg \exists y \bigwedge_{m<\vert \sigma \vert} U^{\sigma(i)}_{i}(y)$
\end{itemize}
It is straightforward to see that $\Th$ is a complete first-order elementary theory. To describe the type of an element of $\Th$, it suffices to describe the single path $\rho\in [T]$ which determines which $U_i$ hold on this element.

Let $\mc A$ be the model of $\Th$ built as follows: We describe the type of the element $\langle x,n, z\rangle$. If there are $n$ numbers $j$ which are $x$-denying, then let $j_n$ be the $n$th of these. Then the type of $\langle x,n, z\rangle$ corresponds to the path $0^x\smallfrown 1 \smallfrown \sigma^x_{j_n}\smallfrown 0^\infty$. If there are not $n$ numbers $j$ which are $x$-denying, then the type of $\langle x,n, z\rangle$ corresponds to the path $0^n\smallfrown 1 \smallfrown \sigma^x$. Note that every path corresponding to a type realized in $\mc A$ is isolated in $T$.

Let $\mc C_x$ be the model of $\Th$ built as follows: $\mc \mc C_x$ is $\mc A$ along with one more element $z_i$ whose type corresponds to $0^x\smallfrown 1 \smallfrown \sigma^x$. Note that $\mc C_x$ is isomorphic to $\mc A$ if and only if the type corresponding to $0^x\smallfrown 1 \smallfrown \sigma^x$ is realized in $\mc A$. This holds if and only if there is some $n$ so that there are \emph{not} $n$ $x$-negating stages, i.e., if and only if there are only finitely many $x$-negating stages, meaning $x\in S$. Thus $\mc C_x\cong \mc A$ if and only if $x\in S$.

We let $\Th'$ be formed from $\Th$ by performing, for each $j$, the $\Sigma^0_{\delta_j}$ Marker extension for the relation $U_j$. Formally, this means that we add a $j$th sort and a relation between the home sort and this sort, and associated to each element $y$ of the home sort, we place either a copy of $(\omega^{\delta_j},<)$ or $(\omega^{\delta_j+1},<)$ in the $j$th sort to encode either $U_j(y)$ or $\neg U_j(y)$.

By the pair of structures theorem \cite{AK90}, for any model $\mc M$ of $\Th$, if the relations $U_j$ are uniformly $\Sigma^0_{\delta_j}(\mathbf{d})$, then $\mathbf{d}$ can compute a copy of the associated extension $\mc{M}'$, which is a model of $\Th'$. Note that $\mc A$ and $\mc C_x$ each satisfy this hypothesis: To determine whether or not $U_j$ holds on an element of $\mc A$ or $\mc C_i$ requires computations from the oracle $\mathbf{0}^{(\delta_j)}$, which is thus uniformly $\Sigma^0_{\delta_j}$. We thus let $\mc {A}'$ and $\mc{C}'_x$ be the uniformly computable sequence of Marker extensions of $\mc{A}$ and $\mc{C}_x$. These are each models of the theory $\Th'$, and $\mc A' \cong \mc C'_x$ if and only if $x\in S$.

It now remains to show that $\mc A'$ has a $\Pinf{\lambda+1}$ Scott sentence. It suffices to say that $\mc{A}'$ satisfies the following sentences:
\begin{itemize}
	\item Each of the extra sorts is a label coming from the pair of structures theorem. This is the conjunction over each $r$ of the sentence saying that each element of the home sort is attached to a copy of $\omega^{\delta_j}$ or $\omega^{\delta_j+1}$. This is $\Pinf\lambda$.
	\item For each $\sigma\in T$, there are infinitely many elements whose type goes through $\sigma$. This is $\Pinf\lambda$.
	\item For each $\sigma\notin T$, there are no elements  whose type goes through $\sigma$. This is $\Pinf\lambda$.
	\item For each $x\notin S$, and each element $y$ in the home sort of $\mc A'$, the type of $y$ does not correspond to $0^x\smallfrown 1\smallfrown \sigma^x$. This is $\Pinf{\lambda+1}$.
\end{itemize}
Note that the conjunction of these statements ensures that $\mathcal{A}'$ is the extension of the prime model of $\Th$, thus determining the isomorphism type of any countable model satisfying it. Thus, the conjunction of these is a $\Pinf{\lambda+1}$ Scott sentence for $\mc A'$, as needed.
\end{proof}

\subsection{Computable Scott Analysis}
We now establish the missing results for  \autoref{FullScottSent}. The first result shows that \autoref{thm:lightfaceScott SentenceHard} is also best possible.
This follows from the following well-known result (see \cite[Lemma VI.14]{Part2}), but we include a short proof here for completeness.

\begin{lem}\label{lem:2alphacomp}\label{55}
	If $\mc{A}$ is computable and has a $\Pinf{\alpha}$ Scott sentence, then it has a $\Pinc{2\alpha}$ Scott sentence.
\end{lem}

\begin{proof}
	If $\mc{A}$ is computable, then Lemma \ref{computable version defining back-and-forth by Pi 2 alpha} gives a $\Pinc{2\alpha}$ sentence $\psi$ such that for any structure $\mc B$,
	\[\mc{B}\models\psi \iff \mc{B}\geq_\alpha\mc{A}.\]
	This means that, if $\mc{A}$ has a $\Pi_\alpha$ Scott sentence $\varphi$, then $\psi$ is also a Scott sentence for $\mc A$ since
	\[\mc{B}\models\psi\implies \mc{B}\geq_\alpha\mc{A}\implies \mc{B}\models \varphi \implies \mc{B}\cong\mc{A}.\qedhere\]
\end{proof}

The following is an immediate corollary of \autoref{thm:lightfaceScott SentenceHard}.

\begin{cor}\label{cor:noSigma2alpha}\label{53}
For each computable $\alpha\geq2$, there is a computable structure with a $\Pinf\alpha$ Scott sentence yet no $\Sinc{2\alpha}$ Scott sentence.
\end{cor}

We can also use \autoref{thm:lightfaceScott SentenceHard} to demonstrate that the oracle $\mathbf 0^{(\gamma(\alpha))}$ used in \autoref{thm:compScottSent} is optimal.

\begin{cor}\label{cor5.4}\label{54}
For each computable ordinal $\alpha\geq 2$,
there is a computable structure $\mc{A}$ with a $\Pinf\alpha$ Scott sentence but no $\Pi^{\mathbf0^{(\delta)}}_\alpha$ Scott sentence for any $\delta<\gamma(\alpha)$.
\end{cor}

\begin{proof}
Let $\mc{A}$ be as in \autoref{thm:lightfaceScott SentenceHard}.
Suppose that $\delta<\gamma(\alpha)$ and that there were some $\Pi_{\alpha}^{\mathbf{0}^{(\delta)}}$ Scott sentence sentence $\chi$ for $\mc{A}$.	
Then $\mathcal{C}\cong \mathcal A$ if and only if $\mathcal{C}\models \chi$, but this is a condition which is $\Pi^0_{\alpha}(\mathbf{0}^{(\delta)})$.
Since $\delta<\gamma(\alpha)$, we have $\delta+\alpha<2\alpha$.
Thus, the structures isomorphic to $\mc{A}$ form a $\Pi^0_{\delta+\alpha}$ set, contradicting its $\Pi^0_{2\alpha}$-completeness.
\end{proof}

 \autoref{thm:compScottSent},  \autoref{lem:2alphacomp}, and their optimality interplay in an interesting manner.
Both results point to the structure $\mc{A}$ from  \autoref{thm:lightfaceScott SentenceHard} being optimally misbehaved from the point of view of effective Scott analysis.
That said, the behavior of the effective Scott sentence of $\mc{A}$ differs depending on whether you would like to optimize quantifier alternation and sacrifice the sentence description or you would like to optimize the sentence description and sacrifice quantifier alternation.
The former analysis suggests that the index set of computable copies of $\mc A$ is $\Pi^0_{\gamma(\alpha)+\alpha}$, which is usually far greater than the reality of $\Pi^0_{2\alpha}$.
A bit more thought does explain this discrepancy, though.
The index set governed by a  $\Pi_{\alpha}^{\mathbf{0}^{(\delta)}}$ formula is of the form $\Pi^0_{\delta+\alpha}$.
The ordinal $\gamma(\alpha)$ is exactly the least ordinal that moves $\delta+\alpha$, so even though it may be far greater than $\Pi^0_{2\alpha}$ it is still the least $\delta$ for which $\Pi^0_{\delta+\alpha}$ eclipses $\Pi^0_{2\alpha}$.
For this reason, both optimality results can exist simultaneously, even though they may point in different directions at first blush.
Note that this differs from the finite case, where structures have both a $\Pinc{2n}$ Scott sentence and a $\mathbf{0}^{(n)}$-computable $\Pinf n$ Scott sentence, both of which point to a $\Pi^0_{2n}$ index set result (see \cite[Section 3]{mhtArith}).

\subsection{Back-and-forth relations}

In this section, we examine the complexity of the back-and-forth relations and establish the missing results in \autoref{FullBNF}.

We use the belligerent pairs theorem to show that the oracle $\mathbf 0^{(\gamma(\alpha))}$ used in  \autoref{bfDef} is optimal.
This is an example of a direct application of the Belligerent Pairs Theorem that does not require the entire machinery of Belligerent Jump Inversion.

\begin{thm}\label{sharpGamma}
	The parameter $\gamma(\alpha)$ in \autoref{bfDef} (or  \autoref{cor:twobnfDef}) is sharp. That is, there exists a computable list $(\mc A_i)_{i\in \omega}$ of computable structures (containing only two isomorphism classes of structures) containing a structure $\mc A$ so that for every $\delta<\gamma(\alpha)$, there is no $\Pi_\alpha^{{\mathbf 0}^{(\delta)}}$ sentence $\phi^\alpha$ so that $\mc A \leq_\alpha \mc A_j \Leftrightarrow \mc A_j\models \phi^\alpha$.
\end{thm}
\begin{proof}
	For the limit ordinal case, $\gamma(\alpha)=0$, so it is clearly sharp. If $\alpha=\beta+1$, then let $\mathcal{A}$, $\mc B$, and $\mc C_i$ be the structures guaranteed by \autoref{UnfriendlyPairs} for a $\Sigma^0_{2\beta+1}$-complete set.
	Suppose that $\delta<\gamma(\alpha)$ and that there were some  $\Pi_{\beta+1}^{\mathbf{0}^{(\delta)}}$-sentence $\psi$ 	
	so that $\mathcal{A}\models \psi$ and $\mathcal{B}\models \neg \psi$. Then $\mathcal{C}_i\cong A$ if and only if $\mathcal{C}_i\models \psi$, but this is a condition which is $\Pi^0_{\beta+1}(\mathbf{0}^{(\delta)})$. But since $\delta<\gamma(\alpha)$ this is a $\Pi^0_{\delta+\beta+1}=\Pi^0_{\beta+1}$-condition. Thus $S$ is $\Pi^0_{\beta+1}$, which contradicts $S$ being $\Sigma^0_{2\beta+1}$-complete, since $2\beta+1\geq \beta+1$.
\end{proof}

\autoref{512} showed that whenever $\mc A \not \leq_{\alpha+1}\mc B$, then there exists a $\Pinc{2\alpha+1}$-sentence $\phi$ so that $\mc A\models \phi$ and $\mc B\models \neg\phi$. We now show that the syntactic complexity $\Pinc{2\alpha+1}$ is optimal.

\begin{thm}\label{lightfaceSep}\label{513}
	Fix a computable ordinal $\alpha$.
	There are computable structures $\mc{A}$ and $\mc{B}$ such that $\mc{B}\not\leq_{\alpha+1}\mc{A}$, but for any  $\Sinc{2\alpha+1}$ sentence $\psi$,
	\[\mc{B}\models\psi\implies\mc{A}\models\psi.\]
\end{thm}

\begin{proof}
	Fix $S$ a $\Sigma_{2\alpha+1}^0$ complete set and let $\mc{A}$, $\mc{B}$, and $\mc{C}_i$ be as in \autoref{thm:Inf4.2}.
	The fact that $\mc{B}\not\leq_{\alpha+1}\mc{A}$ is immediate from the theorem statement.
	Suppose towards a contradiction that there were some $\Sinc{2\alpha+1}$ sentence $\psi$ such that $\mc{B}\models\psi$ yet $\mc{A}\models\lnot\psi$.
	Then
	\[S=\{i~|~\mc{C}_i\cong\mc{A}\} = \{i~|~ \mc{C}_i\models \lnot\psi\}.\]
	Note that the last set is $\Pi_{2\alpha+1}^0,$ a contradiction to the completeness of $S$.
\end{proof}

We note that this asymmetry between the  $\Pinc{2\alpha+1}$ and $\Sinc{2\alpha+1}$ sentences separating $\mc B\not\leq_{\alpha+1} \mc A$ does not exist under the assumption that $\mc B\not\leq_{\lambda} \mc A$ for $\lambda$ a limit. 

\begin{prop}
	Let $\lambda$ be a computable limit ordinal and $\mc{A}$ and $\mc{B}$ be computable structures.
	If $\mc{B}\not\leq_{\lambda}\mc{A}$, then there is a computable $\Sigma_{<\lambda}^{\text{in}}$ sentence $\psi$ such that $\mc{B}\models\psi$ yet $\mc{A}\models\lnot\psi$.
\end{prop}

\begin{proof}
	If $\mc{B}\not\leq_{\lambda}\mc{A}$, there must be some $\delta<\lambda$ such that $\mc{B}\not\leq_{\delta}\mc{A}$.
	Consider the set
	\[A_\delta:=\{\mc{C}~|~\mc{C}\geq_\delta\mc{A}\}.\]
	Because $\mc{A}$ is computable, $A_\delta$ is $\Pinc{2\delta}$-definable by Lemma \ref{computable version defining back-and-forth by Pi 2 alpha}.
	Note that $2\delta$ is only finitely far from $\delta$, so remains less than $\lambda$.
	However, it is clear that $\mc{A}\in A_\delta$ yet $\mc{B}\notin A_\delta.$
	Thus, there is a $\Sinc{2\delta+1}$ formula $\psi$ which defines the complement of $A_\delta$ so $\mc{B}\models\psi$ yet $\mc{A}\models\lnot\psi$, as desired.
\end{proof}

We now see another applications of infinite belligerent jump inversion that generalize results in \cite{mhtArith}.
We do not follow Harrison-Trainor's proof, preferring to use \autoref{thm:lightfaceScott SentenceHard} directly.
\begin{thm}
	For each computable $\alpha$ there is a computable structure $\mc{M}$ such that the relation
	\[\{(\bar{a},\bar{b})\in \mc{M}^{<\omega}\times \mc{M}^{<\omega} | \bar{a}\leq_\alpha\bar{b}\},\]
	is $\Pi^0_{2\alpha}$-complete.
\end{thm}

\begin{proof}
	For $\mc A$ computable, $\{(\bar{a},\bar{b})\in \mc{M}^{<\omega}\times \mc{M}^{<\omega} | \bar{a}\leq_\alpha\bar{b}\}$ is $\Pi^0_{2\alpha}$ by \autoref{computable version defining back-and-forth by Pi 2 alpha}, so we need only find $\mc A$ computable so that this set is $\Pi^0_{2\alpha}$-hard.
	Let $\mc A$ and $\mc{C}_i$ be the computable sequence of structures from \autoref{thm:lightfaceScott SentenceHard} so $\mc A$ has $\Pinf \alpha$ Scott sentence and $\mc{C}_i\cong \mc A$ if and only if $i\in S$, where $S$ is a $\Pi^0_{2\alpha}$-complete set. Let $\mc L$ be the common language of $\mc A$ and $\mc{C}_i$. Let $\mc M$ be the structure in language $\mc{L}\cup \{E\}\cup \{U\}$ where $U$ is unary and defines an infinite set $\{u_i\mid i\in \omega\}$, and $E$ is binary so that the sets $E(u,\mc{M})$ for $u\in U(\mc M)$, partition $M \smallsetminus U(\M)$ into infinite sets. Finally, let $\mc{C}_i$ appear as an $\mc L$-structure on the set  $E(u_i,\mc{M})$.
	
	Fix $u_i$ so $\mc{C}_i\cong \mc A$.
	For any $u_j$, if $u_i\leq_\alpha u_j$ holds then the $\mc L$-structure on $E(u_i,\mc M)$ must be $\leq_\alpha$ the $\mc L$-structure on $E(u_j,\mc M)$. Thus $u_i\leq_\alpha u_j\implies \mc{C}_j\cong \mc A$. On the other hand, if $\mc{C}_j\cong \mc A$, then there is an automorphism of $\mc M$ moving $u_i$ to $u_j$, so $u_i\leq_\alpha u_j$.
	Thus $u_i\leq_\alpha u_j$ if and only if $j\in S$, which is $\Pi^0_{2\alpha}$-hard.
\end{proof}

\subsection{Other applications}

We now use these techniques to give a complete answer to a question of Chen, Gonzalez, and Harrison-Trainor  \cite[Question 1.4]{CGHT} regarding the complexity of being  $\alpha$-back-and-forth above a structure.

\ThmMainC*

\begin{proof}
	The fact that for any computable $\mc A$, $\{\mc{B} | \mc{A}\leq_\alpha \mc{B}\} $ is lightface $\Pi^0_{2\alpha}$ follows from Lemma \ref{computable version defining back-and-forth by Pi 2 alpha}.
	Let $\mc A$ be as in \autoref{thm:lightfaceScott SentenceHard}, and let $f$ be the map sending $i$ to $ \mc{C}_{i,\alpha}$.
	Suppose towards a contradiction that $\{\mc{B} | \mc{A}\leq_\alpha \mc{B}\} $ were lightface $\Sigma^0_{2\alpha}$. Then the index set of computable structures $\mc B$ so that $\mc{A}\leq_\alpha \mc{B}$ would be $\Sigma^0_{2\alpha}$. On the other hand,
if $i\in S$, then $\mc{C}_{i,\alpha}\cong \mc{A}$, so certainly $\mc{C}_{i,\alpha}\geq_\alpha \mc{A}$,
If $i\not\in S$ then $\mc{C}_{i,\alpha}\not\cong \mc{A}$.
As $\mc{A}$ has a $\Pinf{\alpha}$ Scott sentence, this is the same as saying that  $\mc{C}_{i,\alpha}\not\geq_\alpha\mc{A}$. Thus $f$ witnesses the $\Pi^0_{2\alpha}$-hardness of index set of computable structures $\mc B$ so that $\mc A \leq_\alpha \mc B$, contradicting this set being $\Sigma^0_{2\alpha}$.
\end{proof}

We now examine another application of a similar strategy based on the work of Gonzalez and Knight.
This work aims to understand the complexity of saying that a structure has a particular Scott rank.
The upper bound is given by the following lemma.
\begin{lem}[\cite{Part2}, Lemma II.67]
	For each computable ordinal $\alpha$ and language $\mc L$, there exists a $\Pinc{2\alpha+2}$ sentence $\phi$ so that for any structure $\mc A$ \[\mc A \models \phi \Longleftrightarrow \text{SR}(\mc A)\leq \alpha. \]
	In particular, the set of $\mc L$-structures that have Scott rank at most $\alpha$ is lightface $\Pi_{2\alpha+2}$. Also, the index set of computable $\mc L$-structures that have Scott rank at most $\alpha$ is $\Pi^0_{2\alpha+2}$.
\end{lem}

Gonzalez and Knight demonstrated that the above lemma is best possible in the setting of effective descriptive set theory.
They used a modified version of Harrison-Trianor's jump inversion, more suitable for that setting, alongside other tools more common in that area of study.
In particular, they show the following.

\begin{thm}[\cite{GK}]
	Let $\alpha\ge1$ be a computable ordinal.
	The set of structures of Scott rank $\leq \alpha$ is lightface $\Pi^0_{2\alpha+2}$-hard.

    This means that given a computable ordinal $\alpha$, and a lightface $\Pi^0_{2\alpha+2}$ subset $S$ of $\omega^\omega$, there is a Turing-computable operator $F:\omega^\omega\to Mod(Graphs)$ such that
\[x\in S \iff SR(F(x))\leq \alpha. \]
\end{thm}

Their methods blend index set computations into their proof strategy.
In particular, they show the following results about index sets.

\begin{thm}[{\cite[Lemma 4.4 and Theorem 4.11]{GK}}]\label{thm:SRHardAtLimits}
Let $\alpha=1$ or a limit ordinal.
Let $S$ be a complete $\Pi^0_{2\alpha+2}$ set.
There is a uniformly computable sequence of structures $\mc{C}_i$ such that
\[i\in S\Rightarrow SR(\mc{C}_i)=\alpha\] and
\[i\notin S\Rightarrow SR(\mc{C}_i)>\alpha\]
\end{thm}

Applying the above methods, we obtain the following general index set analog of Gonzalez and Knight's result.

\begin{thm}
	Let $\alpha\ge1$ be a computable ordinal.
	The index set of computable structures of Scott rank $\leq \alpha$ is $\Pi^0_{2\alpha+2}$-hard.

This means that given a complete $\Pi^0_{2\alpha+2}$ set $S$,
there is a uniformly computable sequence of structures $\mc{C}_i$ such that
\[i\in S\Rightarrow SR(\mc{C}_i)=\alpha\] and
\[i\notin S\Rightarrow SR(\mc{C}_i)>\alpha\]
\end{thm}

\begin{proof}
If $\alpha$ is a limit ordinal, apply \autoref{thm:SRHardAtLimits} directly.
Otherwise, $\alpha=\beta+1$.
Write $S$ as a $\Pi^0_4(\mathbf{0}^{(2\beta+1)})$ set.
Relativize the construction for $\alpha=1$ to $\mathbf{0}^{(2\beta+1)}$ to obatin a uniformly $\mathbf{0}^{(2\beta+1)}$-computable sequence $\mc{D}_i$ such that
\[SR(\mc{D}_i)=1 \iff i\in S.\]
Use \autoref{thm:Inf4.3} to obtain a uniformly computable sequence $\mc{C}_i=\Inv{\beta}{\mc{D}_i}$.
By \autoref{thm:Inf4.3} (4), the Scott complexity of $\mc{C}_i$ is exactly $\beta$ more than the Scott complexity of $\mc{D}_i$. Thus $i\in S$ implies that the Scott rank of $D_i$ is $1$, thus the Scott rank of $C_i$ is $\beta+1=\alpha$. Similarly, $i\notin S$ implies that the Scott rank of $D_i$ is $>1$, thus the Scott rank of $C_i$ is $\beta+\text{SR}(D_i)>\beta+1=\alpha$.
\end{proof}

\subsection{Optimality of the Belligerent Pairs Theorem and the Belligerent Jump Inversion Theorem	}

We can now show that our main technical theorems \autoref{thm:Inf4.2} and \autoref{thm:Inf4.3} are best possible.
This is because we can show that the applications above are as good as possible, and any improvement to  \autoref{thm:Inf4.2} and \autoref{thm:Inf4.3} would lead to a commensurate and contradictory improvement in the application.
\begin{prop}\label{prop:bestPossible}
	\begin{enumerate}
		\item For each $X$-computable ordinal $\beta$, there is no uniform procedure $\mc{G}\mapsto \hat{\mc{G}}^{(-\beta)}$ such that a) every $X^{(2\beta+2)}$ computable copy of $\hat{\mc{G}}^{(-\beta)}$ uniformly yields an $X$-computable copy of $\mc{G}$ and b) $SR(\mc{G})\leq \delta$ implies $SR(\hat{\mc{G}}^{(-\beta)})\leq \beta+\delta$.
		\item For each $X$-computable ordinal $\alpha$, there is no $X$-computable sequence of structures $\mc{A},\mc{B},\mc{C}_i$ with $\mc{A} \ncong \mc{B}$ such that for $S \subseteq \mathbb{N}$ a complete $\Sigma^0_{2\alpha+2}(X)$ set,
		\[ i \in S \Longrightarrow \mc{C}_i \cong \mc{A} \]
		and
		\[ i \notin S \Longrightarrow \mc{C}_i \cong \mc{B}.\]
		so that:
		\begin{enumerate}
			\item $\mc{B} \nleq_{\alpha+1} \mc{A}$
			\item $\mc{A}$ and $\mc{B}$ have $\Pinf{\alpha+2}$ Scott sentences.
		\end{enumerate}
	\end{enumerate}
	
\end{prop}

\begin{proof}
	We begin with the first claim. We let $X=\emptyset$ and relativize below. Suppose towards a contradiction that such a uniform procedure $\mc{G}\mapsto \hat{\mc{G}}^{(-\beta)}$ existed for some $\beta$. Let $\alpha=\beta+2$. In \autoref{thm:lightfaceScott SentenceHard}, we produce a computable structure $\mc A$ with a $\Pinf\alpha$ Scott sentence, yet a sequence of computable structures so that isomorphism with $\mc A$ is $\Pi^0_{2\alpha}$-complete. We do this by relativizing the analogous result for $\alpha=2$ to the degree $\mathbf{0}^{(2\beta+1)}$ and applying the $\beta$-belligerent jump inversion $\mc G\mapsto \Inv{\beta}{\mc{G}}$. Supposing that we had the map $\mc{G}\mapsto \hat{\mc{G}}^{(-\beta)}$, we instead relativize to $\mathbf{0}^{(2\beta+2)}$ and apply the map $\mc{G}\mapsto \hat{\mc{G}}^{(-\beta)}$ instead. This yields a computable structure $\mc A$ and a computable sequence $\mc{C}_i$ so that $\mc{C}_i\cong \mc A$ if and only if $i\in S$ where $S$ is a $\Pi^0_{2\beta+5}$-complete set. Further, $\mc A$ has a $\Pinf{\alpha}$ Scott sentence. But \autoref{computable version defining back-and-forth by Pi 2 alpha} implies that $\{i | \mc{C}_{i}\cong\mc{A}\}$ is a $\Pi^0_{2\alpha+4}$ set, which is a contradiction to $S$ being a $\Pi^0_{2\beta+5}$-complete set. We can relativize this entire argument to any set $X$ and $\beta$ computable in $X$.

	We now analyze the second claim.
	This follows from the first claim along with the construction specified in the proof of \autoref{thm:Inf4.3}.
	To be more sepcific, given such a sequence $\mc{A},\mc{B},\mc{C}_i$ and an $X^{2\alpha+2}$ computable graph $\mc{G}$, we obtain a map $\mc{G}\mapsto \hat{\mc{G}}^{(-\alpha)}$ defined by replacing each edge with the pair $\langle\mc{A},\mc{B}\rangle$ and each non-edge with the pair $\langle\mc{B},\mc{A}\rangle$.
	It follows from the properties of $\mc{A},\mc{B},\mc{C}_i$ along with the exact argument of \autoref{thm:Inf4.3} that $\mc{G}\mapsto \hat{\mc{G}}^{(-\alpha)}$ is as in the first claim.
	This is a contradiction to the first claim.
	Therefore, no such sequence $\mc{A},\mc{B},\mc{C}_i$ can exists.
\end{proof}

	\bibliographystyle{alpha}
	\bibliography{JumpInversions}

\end{document}